\documentclass[a4paper,11pt,final]{scrartcl}
\usepackage[utf8]{inputenc}
\usepackage{ifpdf}
\usepackage{latexsym}
\usepackage{exscale}
\usepackage{amsfonts}
\usepackage{eucal}
\usepackage{cmmib57}
\usepackage[centertags]{amsmath}
\usepackage{amssymb}
\usepackage{amsbsy}
\usepackage{amstext}
\usepackage{amsthm,thmtools}
\usepackage{xcolor}
\usepackage{framed}
\usepackage{tocbibind}
\usepackage{tocloft}
\usepackage{etoolbox}
\usepackage[ruled,vlined,algosection]{algorithm2e}
\usepackage{delarray}
\usepackage[automark]{scrpage2}
\usepackage{times}
\usepackage[toc,page]{appendix}
\usepackage{textcomp}
\usepackage{mflogo}
\usepackage{here}
\usepackage[round]{natbib}
\usepackage{here}
\usepackage{longtable}
\usepackage{threeparttablex}
\usepackage{multirow}
\usepackage{multicol}
\usepackage{caption}
\usepackage{lscape}
\usepackage{rotating,capt-of}
\usepackage[margin=15mm]{geometry}
\usepackage{tikz}
\usetikzlibrary{shapes}

\newcounter{treeline}

\ifvtexpdf\pdftrue\fi
\ifpdf
\usepackage{graphicx}
\usepackage[pdftex]{thumbpdf}
\usepackage{nameref}
\usepackage[pdftex,colorlinks]{hyperref}
\hypersetup{
           breaklinks=true,
           pdfpagemode=UseOutlines,
           colorlinks=true,
           linkcolor=blue,
           citecolor=blue,
           bookmarksopen=true,
           pdftitle={On the Single-Valuedness of the Pre-Kernel},
           pdfauthor={Holger Ingmar Meinhardt},
           pdfsubject={Single-Valuedness of the Pre-Kernel, Tu Games},
           pdfkeywords={Fairness, Pre-Kernel, Pre-Nucleolus, Tu Games},
}

\else
\usepackage{epsfig}
\usepackage[ps2pdf]{thumbpdf}
\usepackage{nameref}
\usepackage[ps2pdf]{hyperref}

\hypersetup{
           breaklinks=true,
           pdfpagemode=UseOutlines,
           colorlinks=true,
           linkcolor=blue,
           citecolor=blue,
           bookmarksopen=true,
           pdftitle={On the Single-Valuedness of the Pre-Kernel},
           pdfauthor={Holger Ingmar Meinhardt},
           pdfsubject={Single-Valuedness of the Pre-Kernel, Tu Games},
           pdfkeywords={Fairness, Pre-Kernel, Pre-Nucleolus, Tu Games},
}
\fi

\addtocounter{MaxMatrixCols}{20}

\makeatletter
\renewcommand{\section}{\@startsection
  {section}%
  {1}%
  {0em}%
  {-\baselineskip}%
  {0.5\baselineskip}%
  {\centering\normalfont\Large\scshape\mdseries}}%
\makeatother

\makeatletter
\renewcommand{\subsection}{\@startsection
  {subsection}%
  {2}%
  {0em}%
  {-\baselineskip}%
  {0.5\baselineskip}%
  {\normalfont\large\scshape\mdseries}}%
\makeatother

\makeatletter
\renewcommand*\env@matrix[1][c]{\hskip -\arraycolsep
  \let\@ifnextchar\new@ifnextchar
  \array{*\c@MaxMatrixCols #1}}
\makeatother

\makeatletter

\newenvironment{theopargself*}
    {\def\@spopargbegintheorem##1##2##3##4##5{\trivlist
         \item[\hskip\labelsep{##4##1\ ##2}]{\hspace*{-\labelsep}##4##3\@thmcounterend}##5}
     \def\@Opargbegintheorem##1##2##3##4{##4\trivlist
         \item[\hskip\labelsep{##3##1}]{\hspace*{-\labelsep}##3##2\@thmcounterend}}}{}
\def \@floatboxreset {%
        \reset@font
        \small
        \@setnobreak
        \@setminipage
}
\def\figure{\@float{figure}}
\@namedef{figure*}{\@dblfloat{figure}}
\def\table{\@float{table}}
\@namedef{table*}{\@dblfloat{table}}
\def\fps@figure{htbp}
\def\fps@table{htbp}

\makeatother

\renewcommand{\thetable}{\thesection.\arabic{table}}

\makeatletter
\@addtoreset{equation}{section}
\makeatother
\makeatletter
\@addtoreset{table}{section}
\makeatother

\theoremstyle{plain}
\newtheorem{theorem}{Theorem}[section]
\newtheorem{proposition}{Proposition}[section]
\newtheorem{corollary}{Corollary}[section]
\newtheorem{lemma}{Lemma}[section]
\newtheorem{definition}{Definition}[section]

\newtheoremstyle{break}
  {9pt}
  {9pt}
  {\itshape}
  {}
  {\bfseries}
  {.}
  {\newline}
  {}

\newtheoremstyle{break1}
  {9pt}
  {9pt}
  {\rmfamily}
  {}
  {\scshape}
  {.}
  {\newline}
  {}

\theoremstyle{break}

\newtheoremstyle{note}
  {3pt}
  {3pt}
  {}
  {}
  {\itshape}
  {:}
  {.5em}
  {\newline}  
  {}

\theoremstyle{note}

\theoremstyle{definition}
\newtheorem{example}{Example}[section]

\theoremstyle{break1}

\begin{document}
\bibliographystyle{plainnat} 
\pdfbookmark[0]{On the Single-Valuedness of the Pre-Kernel}{tit}
\title{{On the Single-Valuedness of the Pre-Kernel} 
}
\author{{\bfseries Holger I. MEINHARDT} 
~\thanks{Holger I. Meinhardt, Institute of Operations Research, Karlsruhe Institute of Technology (KIT), Englerstr. 11, Building: 11.40, D-76128 Karlsruhe. E-mail: \href{mailto:Holger.Meinhardt@wiwi.uni-karlsruhe.de}{Holger.Meinhardt@wiwi.uni-karlsruhe.de}} 
}
\maketitle

\begin{abstract}
Based on results given in the recent book by~\citet{mei:13}, which presents a dual characterization of the pre-kernel by a finite union of solution sets of a family of quadratic and convex objective functions, we could derive some results related to the uniqueness of the pre-kernel. Rather than extending the knowledge of game classes for which the pre-kernel consists of a single point, we apply a different approach. We select a game from an arbitrary game class with a single pre-kernel element satisfying the non-empty interior condition of a payoff equivalence class, and then establish that the set of related and linear independent games which are derived from this pre-kernel point of the default game replicates this point also as its sole pre-kernel element. In the proof we apply results and techniques employed in the above work. Namely, we prove in a first step that the linear mapping of a pre-kernel element into a specific vector subspace of balanced excesses is a singleton. Secondly, that there cannot exist a different and non-transversal vector subspace of balanced excesses in which a linear transformation of a pre-kernel element can be mapped. Furthermore, we establish that on the restricted subset on the game space that is constituted by the convex hull of the default and the set of related games, the pre-kernel correspondence is single-valued, and therefore continuous. Finally, we provide sufficient conditions that preserve the pre-nucleolus property for related games even when the default game has not a single pre-kernel point.\\


\end{abstract}


\thispagestyle{empty}
\pagebreak

\pagestyle{scrheadings}  \ihead{\empty} \chead{On the Single-Valuedness of the Pre-Kernel} \ohead{\empty}

\section{Introduction}

The coincidence of the kernel with the nucleolus -- that is, the kernel consists of a single point -- is only known for some classes of transferable utility games. In particular, it was established by~\citet{MPSh:72} that for the class of convex games -- introduced by~\citet{Shapley:71} -- the kernel and the nucleolus coincide. Recently,~\citet{getraf:12} were able to extend this result to the class of zero-monotonic almost-convex games. However, for the class of average-convex games, there is only some evidence that both solution concepts coalesce. 

In order to advance our understanding about TU games and game classes which possess an unique pre-kernel element, we propose an alternative approach to investigate this issue while applying results and techniques recently provided in the book by~\citet{mei:13}. There, it was shown that the pre-kernel of the grand coalition can be characterized by a finite union of solution sets of a family of quadratic and convex functions (Theorem 7.3.1). This dual representation of the pre-kernel is based on a Fenchel-Moreau generalized conjugation of the characteristic function. This generalized conjugation was introduced by~\citet{mart:96}, which he called the indirect function. Immediately thereafter, it was~\citet{mes:97} who proved that the pre-kernel can be derived from an over-determined system of non-linear equations. This over-determined system of non-linear equations is equivalent to a minimization problem, whose set of global minima is equal to the pre-kernel set. However, an explicit structural form of the objective function that would allow a better and more comprehensive understanding of the pre-kernel set could not be performed.

The characterization of the pre-kernel set by a finite union of solution sets was possible due to a partition of the domain of the objective function into a finite number of payoff sets. From each payoff vector contained into a particular payoff set the same quadratic and convex function is induced. The collection of all these functions on the domain composes the objective function from which a pre-kernel element can be singled out. Moreover, each payoff set creates a linear mapping that maps payoff vectors into a vector subspace of balanced excesses. Equivalent payoff sets which reflects the same underlying bargaining situation produce the same vector subspace. The vector of balanced excesses generated by a pre-kernel point is contained into the vector subspace spanned by the basis vectors derived from the payoff set that contains this pre-kernel element. In contrast, the vectors of unbalanced excesses induced from the minima of a quadratic function do not belong to their proper vector subspace. An orthogonal projection maps these vectors on this vector subspace of the space of unbalanced excesses (cf.~\citet[Chap. 5-7]{mei:13}). 

From this structure a replication result of a pre-kernel point can be attained. This is due that from the payoff set that contains the selected pre-kernel element, and which satisfies in addition the non-empty interior condition, a null space in the game space can be identified that allows a variation within the game parameter without affecting the pre-kernel properties of this payoff vector. Even though the values of the maximum surpluses have been varied, the set of most effective coalitions remains unaltered by the parameter change. Hence, a set of related games can be determined, which are linear independent, and possess the selected pre-kernel element of the default game as well as a pre-kernel point (cf.~\citet[Sect. 7.6]{mei:13}). In the sequel of this paper, we will establish that the set of related games, which are derived from a default game exhibiting a singleton pre-kernel, must also possess the same unique pre-kernel, and therefore coincides with the pre-nucleolus.  Notice, that these games need not necessarily be convex, average-convex, totally balanced, or zero-monotonic. They could belong to different subclasses of games, however, they must satisfy the non-empty interior condition. Moreover, we show that the pre-kernel correspondence in the game space restricted to the convex hull that is constituted by the extreme points, which are specified by the default and related games, is single-valued, and therefore continuous.   

The structure of the paper is organized as follows: In the Section~\ref{sec:prel} we introduce some basic notations and definitions to investigate the coincidence of the pre-kernel with the pre-nucleolus. Section~\ref{sec:dprk} provides the concept of the indirect function and gives a dual pre-kernel representation in terms of a solution set. In the next step, the notion of lexicographically smallest most effective coalitions is introduced in order to identify payoff equivalence classes on the domain of the objective function from which a pre-kernel element can be determined. Moreover, relevant concepts from~\citet{mei:13} are reconsidered. Section~\ref{sec:siva} studies the uniqueness of the pre-kernel for related games. However, Section~\ref{sec:lhc} investigates the continuity of the pre-kernel correspondence. In Section~\ref{sec:prspn} some sufficient conditions are worked out under which the pre-nucleolus of a default game can preserve the pre-nucleolus property for related games. A few final remarks close the paper.

\section{Some Preliminaries}
\label{sec:prel}
A cooperative game with transferable utility is a pair $\langle N,v \rangle $, where $N$ is the non-empty finite player set $N := \{1,2, \ldots, n\}$, and $v$ is the characteristic function $v: 2^{N} \rightarrow \mathbb{R}$ with $v(\emptyset):=0$. A player $i$ is an element of $N$, and a coalition $S$ is an element of the power set of $2^{N}$. The real number $v(S) \in \mathbb{R}$ is called the value or worth of a coalition $S \in 2^{N}$. Let $S$ be a coalition, the number of members in $S$ will be denoted by $s:=|S|$. We assume throughout that $v(N) > 0$ and $n \ge 2$ is valid. In addition, we identify a cooperative game by the vector $v := (v(S))_{S \subseteq N} \in \mathcal{G}^{n} = \mathbb{R}^{2^{n}}$, if no confusion can arise. Finally, the relevant game space for our investigation is defined by $\mathcal{G}(N) := \{v \in \mathcal{G}^{n}\,\arrowvert\, v(\emptyset) = 0 \land v(N) > 0\}$.  

If $\mathbf{x} \in \mathbb{R}^{n}$, we apply $x(S) := \sum_{k \in S}\, x_{k}$ for every $S \in 2^{N}$ with $x(\emptyset):=0$. The set of vectors $\mathbf{x} \in \mathbb{R}^{n}$ which satisfies the efficiency principle $v(N) = x(N)$ is called the {\bfseries pre-imputation set} and it is defined by 
\begin{equation} 
  \label{eq:pre-imp}
  \mathcal{I}^{\,0}(v):= \left\{\mathbf{x} \in \mathbb{R}^{n} \;\arrowvert\, x(N) = v(N) \right\}, 
\end{equation} 
where an element $\mathbf{x} \in \mathcal{I}^{\,0}(v)$ is called a pre-imputation. 

Given a vector $\mathbf{x} \in \mathcal{I}^{\,0}(v)$, we define the {\bfseries excess} of coalition $S$ with respect to the pre-imputation $\mathbf{x}$ in the game $\langle N,v \rangle $ by 
\begin{equation} 
  \label{eq:exc} 
  e^{v}(S,\mathbf{x}):= v(S) - x(S). 
\end{equation} 

Take a game $v \in \mathcal{G}^{n}$. For any pair of players $i,j \in N, i\neq j$, the {\bfseries maximum surplus} of player $i$ over player $j$ with respect to any pre-imputation $\mathbf{x} \in \mathcal{I}^{\,0}(v)$ is given by the maximum excess at $\mathbf{x}$ over the set of coalitions containing player $i$ but not player $j$, thus\begin{equation} 
  \label{eq:maxexc} 
  s_{ij}(\mathbf{x},v):= \max_{S \in \mathcal{G}_{ij}} e^{v}(S,\mathbf{x}) \qquad\text{where}\;  \mathcal{G}_{ij}:= \{S \;\arrowvert\; i \in S\; \text{and}\; j \notin S \}. 
\end{equation} 
The set of all pre-imputations $\mathbf{x} \in \mathcal{I}^{\,0}(v)$ that balances the maximum surpluses for each distinct pair of players $i,j \in N, i\neq j$ is called the~\hypertarget{hyp:prk}{{\bfseries pre-kernel}} of the game $v$, and is defined by 
  \begin{equation} 
    \label{eq:prek} 
    \mathcal{P\text{\itshape r}K}(v) := \left\{ \mathbf{x} \in \mathcal{I}^{\,0}(v)\; \arrowvert\;  s_{ij}(\mathbf{x},v) = s_{ji}(\mathbf{x},v) \quad\text{for all}\; i,j \in N, i\neq j \right\}. 
  \end{equation} 

In order to define the pre-nucleolus of a game $v \in \mathcal{G}^{n}$, take any $\mathbf{x} \in \mathbb{R}^{n}$ to define a $2^{n}$-tuple vector $\theta(\mathbf{x})$ whose components are the excesses $e^{v}(S,\mathbf{x})$ of the $2^{n}$ coalitions $S \subseteq N$, arranged in decreasing order, that is,
\begin{equation}
 \label{eq:compl_vec}
  \theta_{i}(\mathbf{x}):=e^{v}(S_{i},\mathbf{x}) \ge e^{v}(S_{j},\mathbf{x}) =:\theta_{j}(\mathbf{x}) \qquad\text{if}\qquad 1 \le i \le j \le 2^{n}.
\end{equation}
Ordering the so-called complaint or dissatisfaction vectors $\theta(\mathbf{x})$ for all $\mathbf{x} \in \mathbb{R}^{n}$ by the lexicographic order  $\le_{L}$ on $\mathbb{R}^{n}$, we shall write
\begin{equation}
 \theta(\mathbf{x}) <_{L} \theta(\mathbf{y}) \qquad\text{if}\;\exists\;\text{an integer}\; 1 \le k \le 2^{n},
\end{equation}
such that $\theta_{i}(\mathbf{x}) = \theta_{i}(\mathbf{y})$ for $1 \le i < k$ and $\theta_{k}(\mathbf{x}) < \theta_{k}(\mathbf{y})$. Furthermore, we write $\theta(\mathbf{x}) \le_{L} \theta(\mathbf{y})$ if either $\theta(\mathbf{x}) <_{L} \theta(\mathbf{y})$ or $\theta(\mathbf{x}) = \theta(\mathbf{y})$. Now the pre-nucleolus $\mathcal{P\text{\itshape r}N}(v)$ over the pre-imputations set $\mathcal{I}^{\,0}(v)$ is defined by 
\begin{equation}
 \label{eq:prn_sol}
  \mathcal{P\text{\itshape r}N}(v) = \left\{\mathbf{x} \in \mathcal{I}^{\,0}(v)\; \arrowvert\; \theta(\mathbf{x}) \le_{L} \theta(\mathbf{y}) \;\forall\; \mathbf{y} \in \mathcal{I}^{\,0}(v) \right\}.
\end{equation}
The {\bfseries pre-nucleolus} of any game $v \in \mathcal{G}^{n}$ is non-empty as well as unique, and it is referred to as $\nu(v)$ if the game context is clear from the contents or $\nu(N,v)$ otherwise. 

\section{A Dual Pre-Kernel Representation}
\label{sec:dprk}

The concept of a Fenchel-Moreau generalized conjugation -- also known as the indirect function of a characteristic function game -- was introduced by~\citet{mart:96}, and provides the same information as the $n$-person cooperative game with transferable utility under consideration. This approach was successfully applied in~\citet{mei:13} to give a dual representation of the pre-kernel solution of TU games by means of solution sets of a family of quadratic objective functions. In this section, we review some crucial results extensively studied in~\citet[Chap.~5 \&~6]{mei:13} as the building blocks to investigate the single-valuedness of the pre-kernel correspondence.   

The {\bfseries convex conjugate} or {\bfseries Fenchel transform} $f^{*}: \mathbb{R}^{n} \to \overline{\mathbb{R}}$ (where $\overline{\mathbb{R}} := \mathbb{R} \cup \{ \pm\;\infty\}$) of a convex function $f: \mathbb{R}^{n} \to \overline{\mathbb{R}}$ (cf.~\citet[Section 12]{Rocka:70}) is defined by 
\begin{equation*} 
  f^{*}(\mathbf{x}^{\,*}) = \sup_{\mathbf{x} \in \mathbb{R}^{n}} \{\langle\; \mathbf{x}^{\,*}, \mathbf{x} \;\rangle - f(\mathbf{x})\} \qquad \forall \mathbf{x}^{\,*} \in \mathbb{R}^{n}.
\end{equation*} 
Observe that the Fenchel transform $f^{*}$ is the point-wise supremum of affine functions $p(\mathbf{x}^{\,*}) = \langle\; \mathbf{x}, \mathbf{x}^{\,*} \;\rangle - \mu$ such that $(\mathbf{x},\mu) \in (\CMcal{C} \times \mathbb{R}) \subseteq (\mathbb{R}^{n} \times \mathbb{R})$, whereas $\CMcal{C}$ is a convex set. Thus, the Fenchel transform $f^{*}$ is again a convex function.

We can generalize the definition of a Fenchel transform (cf.~\citet{mart:96}) by introducing a fixed non-empty subset $\CMcal{K}$ of $\mathbb{R}^{n}$, then the conjugate of a function $f: \CMcal{K} \to \overline{\mathbb{R}}$ is $f^{c}: \mathbb{R}^{n} \to \overline{\mathbb{R}}$, given by
\begin{equation*}
  f^{c}(\mathbf{x}^{\,*}) = \sup_{\mathbf{x} \in \CMcal{K}} \{\langle\; \mathbf{x}^{\,*}, \mathbf{x} \;\rangle - f(\mathbf{x})\} \qquad \forall \mathbf{x}^{\,*} \in \mathbb{R}^{n},
\end{equation*}
which is also known as the {\bfseries Fenchel-Moreau conjugation}.

A vector $\mathbf{x}^{\,*}$ is said to be a subgradient of a convex function $f$ at a point $\mathbf{x}$, if
\begin{equation*}
  f(\mathbf{z}) \ge f(\mathbf{x}) + \langle\; \mathbf{x}^{\,*}, \mathbf{z} - \mathbf{x} \;\rangle \qquad\forall \mathbf{z} \in \mathbb{R}^{n}.
\end{equation*}
The set of all subgradients of $f$ at $\mathbf{x}$ is called the subdifferentiable of $f$ at $\mathbf{x}$ and it is defined by
\begin{equation*}
  \partial f(\mathbf{x}):= \{ \mathbf{x}^{\,*} \in \mathbb{R}^{n}\;\arrowvert\; f(\mathbf{z}) \ge f(\mathbf{x}) + \langle\; \mathbf{x}^{\,*}, \mathbf{z} - \mathbf{x} \;\rangle \quad (\forall \mathbf{z} \in \mathbb{R}^{n})\}.
\end{equation*}
The set of all subgradients $\partial f(\mathbf{x})$ is a closed convex set, which could be empty or may consist of just one point. The multivalued mapping $\partial f: \mathbf{x} \mapsto \partial f(\mathbf{x})$ is called the subdifferential of $f$.

\begin{theorem}[\citet{mart:96}]
\label{th:mart7}
The indirect function $\pi: \mathbb{R}^{n} \to \mathbb{R}$ of any $n$-person TU game is a non-increasing polyhedral convex function such that
\begin{itemize}
\item[(i)] $\partial{\pi(\mathbf{x})}{} \cap \{-1, 0\}^{n} \neq \emptyset \qquad\forall \mathbf{x} \in \mathbb{R}^{n}$,
\item[(ii)] $ \{-1,0\}^{n} \subset \bigcup_{\mathbf{x} \in \mathbb{R}^{n}} \partial{\pi(\mathbf{x})}{}$, and
\item[(iii)] $\min_{\mathbf{x} \in \mathbb{R}^{n}}\; \pi(\mathbf{x}) = 0$.
\end{itemize}
Conversely, if $\pi: \mathbb{R}^{n} \to \mathbb{R}$ satisfies $(i)$-$(iii)$ then there exists an unique $n$-person TU game $\langle N,v \rangle$ having $\pi$ as its indirect function, its characteristic function is given by
\begin{equation}
\label{eq:mart15}
v(S) = \min_{\mathbf{x} \in \mathbb{R}^{n}}\bigg\{\pi(\mathbf{x}) + \sum_{k \in S}\; x_{k}\bigg\} \qquad\forall\; S \subset N.
\end{equation}
\end{theorem}
According to the above result, the associated {\bfseries indirect function} $\pi: \mathbb{R}^{n}\to \mathbb{R}_{+}$ is given by: $$ \pi(\mathbf{x}) = \max_{S \subseteq N}\, \bigg\{v(S) - \sum_{k \in S}\;x_{k}\bigg\},$$ for all $\mathbf{x}\in\mathbb{R}^{n}$. A characterization of the pre-kernel in terms of the indirect function is due to~\citet{mes:97}. Here, we present this representation in its most general form, although we restrict ourselves to the trivial coalition structure $\mathcal{B}=\{N\}$.  

The pre-imputation that comprises the possibility of compensation between a pair of players $i, j \in N, i \neq j$, is denoted as $\mathbf{x}^{\;i,j,\delta} = (x^{\;i,j,\delta}_{k})_{k \in N}\in \mathcal{I}^{\,0}(v)$, with $\delta \ge 0$, which is given by
\begin{equation*}
  \mathbf{x}^{\;i,j,\delta}_{N\backslash\{i,j\}} = \mathbf{x}_{N\backslash\{\;i,j\}},\; x^{i,j,\delta}_{i} = x_{i} - \delta\quad\text{and}\quad x^{\;i,j,\delta}_{j} = x_{j} + \delta.
\end{equation*}

\begin{proposition}[\citet{mes:97,mei:13}]
\label{prop:mese1}
For a TU game with indirect function $\pi$, a pre-imputation $\mathbf{x} \in \mathcal{I}^{\,0}(v)$ is in the pre-kernel of $\langle N,v \rangle$ for the coalition structure $\mathcal{B} = \{B_{1}, \ldots, B_{l} \}$, $\mathbf{x} \in \mathcal{P\text{\itshape r}K}(v,\mathcal{B})$, if, and only if, for every $k \in \{1,2, \ldots, l \}$, every $i,j \in B_{k},\; i < j$, and some $\delta \ge \delta_{1}(v,\mathbf{x})$, one gets
\begin{equation*}
 \pi(\mathbf{x}^{\;i,j,\delta}) = \pi(\mathbf{x}^{\;j,i,\delta}).
\end{equation*}
whereas $\delta_{1}(\mathbf{x},v) := \max_{k \in N, S \subset N\backslash\{k\}}\; |v(S \cup \{k\}) - v(S) - x_{k}|$. 
\end{proposition}

\citet{mes:97} was the first who recognized that based on the result of Proposition~\ref{prop:mese1} a pre-kernel element can be derived as a solution of an over-determined system of non-linear equations. For the trivial coalition structure $\mathcal{B} = \{N\}$ the over-determined system of non-linear equations is given by  
\begin{equation}
  \label{eq:fij} 
  \begin{cases}
    f_{ij}(\mathbf{x}) = 0 & \forall i,j \in N, i < j\\[.5em]
    f_{0}(\mathbf{x}) = 0 
  \end{cases}
\end{equation}
where, for some $\delta \ge \delta_{1}(\mathbf{x},v)$, 
\begin{equation*}
  f_{ij}(\mathbf{x}) := \pi(\mathbf{x}^{\;i,j,\delta}) - \pi(\mathbf{x}^{\;j,i,\delta}) \qquad\forall i,j \in N,i<j,\tag{\ref{eq:fij}-a}
\end{equation*} 
and 
\begin{equation*}
  f_{0}(\mathbf{x}) := \sum_{k \in N}\; x_{k} - v(N).\tag{\ref{eq:fij}-b}
\end{equation*}
To any over-determined system an equivalent minimization problem is associated such that the set of global minima coincides with the solution set of the system (cf.~\citet[Sec. 5.3]{mei:13}). The solution set of such a minimization problem is the set of values for $\mathbf{x}$ which minimizes the following function 
 \begin{equation}
    \label{eq:objfh}
  h(\mathbf{x}) := \sum_{\substack{i,j \in N\\ i < j}}\; (f_{ij}(\mathbf{x}))^2 + (f_{0}(\mathbf{x}))^2 \ge 0 \qquad\;\forall\,\mathbf{x} \in \mathbb{R}^{n}.
\end{equation}
As we will notice in the sequel, this optimization problem is equivalent to a least squares adjustment. For further details see~\citet[Chap. 6]{mei:13}. From the existence of the pre-kernel and objective function $h$ of type~\eqref{eq:objfh}, we get the following relation:

\begin{corollary}[\citet{mei:13}]
 \label{cor:rep}
  For a TU game $\langle N,v \rangle$ with indirect function $\pi$, it holds that 
  \begin{equation}
    \label{eq:prkbyh}
    h(\mathbf{x}) = \sum_{\substack{i,j \in N \\ i < j}}\; (f_{ij}(\mathbf{x}))^2 + (f_{0}(\mathbf{x}))^2 = \min_{\mathbf{y} \in \mathcal{I}^{0}(v)}\; h(\mathbf{y}) = 0,
  \end{equation}
if, and only if, $\mathbf{x} \in \mathcal{P\text{\itshape r}K}(v)$.
\end{corollary}

\begin{proof}
To establish the equivalence between the pre-kernel set and the set of global minima, we have to notice that in view of Theorem~\ref{th:mart7} $\min_{\mathbf{y}} h = 0$ is in force. Now, we prove necessity while taking a pre-kernel element, i.e.~$\mathbf{x} \in \mathcal{P\text{\itshape r}K}(v)$, then the efficiency property is satisfied with $f_{0}(\mathbf{x}) = 0$ and the maximum surpluses $s_{ij}(\mathbf{x}, v)$ must be balanced for each distinct pair of players $i,j$, implying that $f_{ij}(\mathbf{x}) = 0$ for all $i,j \in N, i < j$ and therefore $h(\mathbf{x}) = 0$. Thus, we are getting $\mathbf{x} \in M(h)$. To prove sufficiency, assume that $\mathbf{x} \in M(h)$, then $h(\mathbf{x}) = 0$ with the implication that the efficiency property $f_{0}(\mathbf{x}) = 0$ and $f_{ij}(\mathbf{x}) = 0$ must be valid for all $i,j \in N, i < j$. This means that the difference $f_{ij}(\mathbf{x}) = (\pi(\mathbf{x}^{i,j,\delta}) - \pi(\mathbf{x}^{j,i,\delta}))$  is equalized for each distinct pair of indices $i,j \in N, i < j$. Thus, $\mathbf{x} \in \mathcal{P\text{\itshape r}K}(v)$. It turns out that the minimum set coincides with the pre-kernel, i.e., we have:
\begin{equation}
 \label{eq:prk}
  M(h) = \{\mathbf{x} \in \mathcal{I}^{\,0}(v)\,\arrowvert\; h(\mathbf{x}) = 0 \} = \mathcal{P\text{\itshape r}K}(v), 
\end{equation}
with this argument we are done.
\end{proof}

To understand the structural form of the objective function $h$, we will first identify equivalence relations on its domain. To start with, we define the set of {\bfseries most effective} or {\bfseries significant coalitions} for each pair of players $i,j \in N, i \neq j$ at the payoff vector $\mathbf{x}$ by
\begin{equation}
  \label{eq:bsc_ij}
  \mathcal{C}_{ij}(\mathbf{x}):=\{S \in \mathcal{G}_{ij}\,\arrowvert\, s_{ij}(\mathbf{x},v) = e^{v}(S,\mathbf{x}) \}.
\end{equation}
When we gather for all pair of players $i,j \in N, i \neq j$ all these coalitions that support the claim of a specific player over some other players, we  have to consider the concept of the collection of most effective or significant coalitions w.r.t. $\mathbf{x}$, which we define as in~\citet[p. 315]{MPSh:79} by
\begin{equation}
  \label{eq:bsc}
  \mathcal{C}(\mathbf{x}) := \bigcup_{\substack{i,j \in N \\ i \neq j} } \; \mathcal{C}_{ij}(\mathbf{x}).
\end{equation}
Notice that the set $\mathcal{C}_{ij}(\mathbf{x})$ for all $i,j \in N, i \neq j$ does not have cardinality one, which is required to identify a partition on the domain of function $h$. Now let us choose for each pair $i,j \in N, i \neq j$ a descending ordering on the set of most effective coalitions in accordance with their size, and within such a collection of most effective coalitions having smallest size the lexicographical minimum is singled out, then we obtain the required uniqueness to partition the domain of $h$. This set is denoted by $\mathcal{S}_{ij}(\mathbf{x})$ for all pairs $i,j \in N, i \neq j$, and gathering all these collections we are able to specify the set of lexicographically smallest most effective coalitions w.r.t. $\mathbf{x}$ through
\begin{equation}
  \label{eq:mec}
  \mathcal{S}(\mathbf{x}) := \{ \mathcal{S}_{ij}(\mathbf{x}) \,\arrowvert\, i,j \in N, i \neq j \}.
\end{equation}
This set will be indicated in short as the set of {\bfseries lexicographically smallest coalitions} or just more succinctly {\bfseries most effective coalitions} whenever no confusion can arise. Notice that this set is never empty and can uniquely be identified.  This implies that the cardinality of this set is equal to $n \cdot (n-1)$. In the following we will observe that from these type of sets equivalence relations on the domain $dom\, h$ can be identified.

To see this, consider the correspondence $\mathcal{S}$ on $dom\, h$ and two different vectors, say $\mathbf{x}$ and $\vec{\gamma}$, then both vectors are said to be equivalent w.r.t. the binary relation $\sim$ if, and only if, they induce the same set of lexicographically smallest coalitions, that is, $\mathbf{x} \sim \vec{\gamma}$ if, and only if, $\mathcal{S}(\mathbf{x}) = \mathcal{S}(\vec{\gamma})$. In case that the binary relation $\sim$ is reflexive, symmetric and transitive, then it is an {\bfseries equivalence relation} and it induces {\bfseries equivalence classes} $[\vec{\gamma}]$ on $dom\, h$ which we define through $[ \vec{\gamma} ] := \{ \mathbf{x} \in dom\;h \;\arrowvert \mathbf{x} \sim \vec{\gamma}\}$. Thus, if $\mathbf{x} \sim \vec{\gamma}$, then $[\mathbf{x}] = [\vec{\gamma}]$, and if $\mathbf{x} \nsim \vec{\gamma}$, then $[\mathbf{x}] \cap [\vec{\gamma}] = \emptyset$. This implies that whenever the binary relation $\sim$ induces equivalence classes $[\vec{\gamma}]$ on $dom\, h$, then it partitions the domain $dom\, h$ of the function $h$. The resulting collection of equivalence classes $[\vec{\gamma}]$ on $dom\, h$ is called the quotient of $dom\, h$ modulo $\sim$, and we denote this collection by $dom\, h/\sim$. We indicate this set as an equivalence class whenever the context is clear, otherwise we apply the term payoff set or payoff equivalence class. 

\begin{proposition}[\citet{mei:13}]
  \label{prop:eq_rel}
The binary relation $\sim$ on the set $dom\, h$ defined by $\mathbf{x} \sim \vec{\gamma} \iff \mathcal{S}(\mathbf{x}) = \mathcal{S}(\vec{\gamma})$ is an equivalence relation, which forms a partition of the set $dom\, h$ by the collection of equivalence classes $\{[\vec{\gamma}_{k}]\}_{k \in J}$, where $J$ is an arbitrary index set. Furthermore, for all $k \in J$, the induced equivalence class $[\vec{\gamma}_{k}]$ is a convex set . 
\end{proposition}
\begin{proof}
  For a proof see~\citet[p. 59]{mei:13}.
\end{proof}

The cardinality of the collection of the payoff equivalence classes induced by a TU game is finite (cf.~\citet[Proposition 5.4.2.]{mei:13}). Furthermore, on each payoff equivalence class $[\vec{\gamma}]$ from the $dom\, h$ an unique quadratic and convex function can be identified. Therefore, there must be a finite composite of these functions that constitutes the objective function $h$. In order to construct such a quadratic and convex function suppose that $\vec{\gamma} \in [\vec{\gamma}]$. From this vector we attain the collection of most effective coalitions $\mathcal{S}(\vec{\gamma})$ in accordance with Proposition~\ref{prop:eq_rel}. Then observe that the differences in the values between a pair $\{i,j\}$ of players are defined by $\alpha_{ij} := (v(S_{ij}) - v(S_{ji})) \in \mathbb{R}$ for all $i,j \in N,\, i < j$, and $\alpha_{0} := v(N) > 0$ w.r.t. $\mathcal{S}(\vec{\gamma})$. All of these $q$-components compose the $q$-coordinates of a payoff independent vector $\vec{\alpha}$, with $q = \binom{n}{2} +1$.  A vector that reflects the degree of unbalancedness of excesses for all pair of players, is denoted by $\vec{\xi} \in \mathbb{R}^{q}$, that is a $q$-column vector, which is given by 
\begin{equation}
  \label{eq:unb_exc}
    \begin{split}
  \xi_{ij} & :=  e^{v}(S_{ij},\vec{\gamma}) - e^{v}(S_{ji},\vec{\gamma})  = v(S_{ij}) - \gamma(S_{ij}) - v(S_{ji}) + \gamma(S_{ji}) \quad\forall \, i,j \in N,\, i < j, \allowdisplaybreaks\\
                 & = v(S_{ij}) - v(S_{ji}) + \gamma(S_{ji}) - \gamma(S_{ij})  = \alpha_{ij} + \gamma(S_{ji}) - \gamma(S_{ij})  \quad\forall \, i,j \in N,\, i < j, \allowdisplaybreaks\\
  \xi_{0} & :=  v(N) - \gamma(N) = \alpha_{0}  - \gamma(N). 
 \end{split}
\end{equation}
In view of Proposition~\ref{prop:eq_rel}, all vectors contained in the equivalence~class $[\vec{\gamma}]$ induce the same set $\mathcal{S}(\vec{\gamma})$, and it holds
\begin{equation}
 \label{eq:xi_zet}
  \xi_{ij}  :=  e^{v}(S_{ij},\vec{\gamma}) - e^{v}(S_{ji},\vec{\gamma})  = s_{ij}(\vec{\gamma},v) - s_{ji}(\vec{\gamma},v)  =: \zeta_{ij} \quad\forall \, i,j \in N,\, i < j.
\end{equation}
The payoff dependent configurations $\vec{\xi}$ and $\vec{\zeta}$ having the following interrelationship outside its equivalence class: $\vec{\xi} \neq \vec{\zeta}$ for all $\mathbf{y} \in [\vec{\gamma}]^{c}$. Moreover, equation~\eqref{eq:xi_zet} does not necessarily mean that for $\vec{\gamma}^{\,\prime}, \vec{\gamma}^{*} \in [\vec{\gamma}],\, \vec{\gamma}^{\,\prime} \neq\vec{\gamma}^{*} $, it holds $\vec{\xi}^{\,\prime} = \vec{\xi}^{*}$. Hence, the vector of (un)balanced excesses $\vec{\xi}$ is only equal with the vector of (un)balanced maximum surpluses $\vec{\zeta}$ if the corresponding pre-imputation $\vec{\gamma} $ is drawn from its proper equivalence class $[\vec{\gamma}]$.

In addition, we write for sake of simplicity that $\mathbf{E}_{ij}:= (\mathbf{1}_{S_{ji}} - \mathbf{1}_{S_{ij}}) \in \mathbb{R}^{n}, \;\forall i,j \in N, i < j$, and $\mathbf{E}_{0} := - \mathbf{1}_{N} \in \mathbb{R}^{n}$. Combining these $q$-column vectors, we can construct an $(n \times q)$-matrix in $\mathbb{R}^{n \times q}$ referred to as $\mathbf{E}$, and which is given by
\begin{equation}
\label{eq:matE}
\mathbf{E} := [\mathbf{E}_{1,2}, \ldots ,\mathbf{E}_{n-1,n},\mathbf{E}_{0}]  \in \mathbb{R}^{^{n \times q}}.
\end{equation}

\begin{proposition}[Quadratic Function]
 \label{prop:quad}
Let $\langle N,v \rangle$ be a TU game with indirect function $\pi$, then an arbitrary vector $\vec{\gamma}$ in the domain of $h$, i.e. $\vec{\gamma} \in dom\, h$, induces a quadratic function:
\begin{equation}
 \label{eq:objf2}
h_{\gamma}(\mathbf{x}) = (1/2) \cdot \langle\; \mathbf{x},\mathbf{Q} \,\mathbf{x} \;\rangle + \langle\; \mathbf{x}, \mathbf{a} \;\rangle + \mathbf{\alpha} \qquad \mathbf{x} \in dom\, h,
\end{equation}
where $\mathbf{a}$ is a column vector of coefficients, $\alpha$ is a scalar and $\mathbf{Q}$ is a symmetric ($n \times n$)-matrix with integer coefficients taken from the interval $[-n \cdot (n-1), n \cdot (n-1)]$.
\end{proposition}
\begin{proof}
  The proof is given in~\citet[pp.~66-68]{mei:13}.
\end{proof}
By the above discussion, the objective function $h$ and the quadratic as well as convex function $h_{\gamma}$ of type~\eqref{eq:objf2} coincide on the payoff set $[\vec{\gamma}]$ (cf.~\citet[Lemma 6.2.2]{mei:13}). However, on the complement $[\vec{\gamma}]^{c}$ it holds $h \not= h_{\gamma}$. Moreover, in view of \citet[Proposition 6.2.2]{mei:13} function $h$ is composed of a finite family of quadratic and convex functions of type~\eqref{eq:objf2}. 

\begin{proposition}[Least Squares] 
  \label{prop:eqrep}
A quadratic function $h_{\gamma}$ given by equation~\eqref{eq:objf2} is equivalent to
\begin{equation}
  \label{eq:eqrep}
   \langle\, \vec{\alpha} + \mathbf{E}^{\top}\; \mathbf{x}, \vec{\alpha} + \mathbf{E}^{\top}\; \mathbf{x}\,\rangle = \Arrowvert\, \vec{\alpha} + \mathbf{E}^{\top}\; \mathbf{x}\,\Arrowvert^{2}.
\end{equation}
Therefore, the matrix  $\mathbf{Q} \in \mathbb{R}^{n^2}$ can also be expressed as $\mathbf{Q} = 2 \cdot \mathbf{E} \; \mathbf{E}^{\top}$, and the column vector $\mathbf{a}$ as $2 \cdot \mathbf{E} \; \vec{\alpha} \in \mathbb{R}^{n}$. Finally, the scalar $\alpha$ is given by $\Arrowvert \vec{\alpha} \Arrowvert^2$, where $\mathbf{E} \in \mathbb{R}^{n \times q}, \mathbf{E}^{\top} \in \mathbb{R}^{q \times n}$ and $\vec{\alpha} \in \mathbb{R}^q$.
\end{proposition}
\begin{proof}
  The proof can be found in~\citet[pp.~70-71]{mei:13}.
\end{proof}

Realize that the transpose of a vector or a matrix is denoted by the symbols $\mathbf{x}^{\top}$, and $\mathbf{Q}^{\top}$ respectively.

\begin{lemma}[\citet{mei:13}]
  \label{lem:Etx_Pa}
Let $\mathbf{x}, \vec{\gamma} \in dom\, h,  \mathbf{x} = \vec{\gamma} + \mathbf{z} $ and let $\vec{\gamma}$ induces the matrices $\mathbf{E} \in \mathbb{R}^{n \times q}, \mathbf{E}^{\top} \in \mathbb{R}^{q \times n}$ determined by formula~\eqref{eq:matE}, and $\vec{\alpha}, \vec{\xi} \in \mathbb{R}^q$ as in equation~\eqref{eq:unb_exc}. If $\mathbf{x} \in M(h_{\gamma}) $, then 
\begin{enumerate}
\item $- \mathbf{E}^{\top} \, \mathbf{x}  = \mathbf{P}\, \vec{\alpha} $.
\item $\mathbf{E}^{\top} \, \vec{\gamma}  = \mathbf{P}\, (\vec{\xi} - \vec{\alpha}) = (\vec{\xi} - \vec{\alpha})$.
\item $- \mathbf{E}^{\top} \, \mathbf{z}  = \mathbf{P}\, \vec{\xi}$.
\end{enumerate}
In addition, let $q := \binom{n}{2} + 1$. The matrix $\mathbf{P}\in \mathbb{R}^{q^2}$ is either equal to $2 \cdot \mathbf{E}^{\top}\, \mathbf{Q}^{-1} \mathbf{E}$, if the matrix $\mathbf{Q} \in \mathbb{R}^{n^2}$ is non-singular, or it is equal to $2 \cdot \mathbf{E}^{\top}\, \mathbf{Q}^{\dagger} \mathbf{E}$, if the matrix $\mathbf{Q}$ is singular. Furthermore, it holds for the matrix $\mathbf{P}$ that $\mathbf{P} \neq \mathbf{I}_{q}$ and $\text{rank} \, \mathbf{P} \le n$.
\end{lemma}
\begin{proof}
  The proof is given in~\citet[pp.~80-81]{mei:13}.
\end{proof}
Notice that $\mathbf{Q}^{\dagger}$ is the {\bfseries Moore-Penrose} or {\bfseries pseudo-inverse} matrix of matrix $\mathbf{Q}$, if matrix $\mathbf{Q}$ is singular. This matrix is unique according to the following properties: (1) general condition, i.e. $\mathbf{Q}\,\mathbf{Q}^{\dagger}\,\mathbf{Q} = \mathbf{Q}$, (2) reflexive, i.e. $\mathbf{Q}^{\dagger}\,\mathbf{Q}\,\mathbf{Q}^{\dagger} = \mathbf{Q}^{\dagger}$, (3) normalized, i.e. $(\mathbf{Q}\,\mathbf{Q}^{\dagger})^{\top} = \mathbf{Q}^{\dagger}\,\mathbf{Q}$, and finally (4) reversed normalized, i.e. $(\mathbf{Q}^{\dagger}\,\mathbf{Q})^{\top} = \mathbf{Q}\,\mathbf{Q}^{\dagger}$.

\begin{proposition}[Orthogonal Projection Operator]
  \label{prop:orth_mat}
Matrix $\mathbf{P}$ is idempotent and self-adjoint, i.e. $\mathbf{P}$ is an orthogonal projection operator.
\end{proposition}
\begin{proof}
  The proof can be found in~\citet[p.~86]{mei:13}.
\end{proof}

\begin{lemma}[\citet{mei:13}.]
  \label{lem:spV1}
 Let $\mathcal{E}$ be a subspace of $\mathbb{R}^{q}$ with basis $\{\mathbf{e}_{1}, \ldots, \mathbf{e}_{m}\}$ derived from the linear independent vectors of matrix $\mathbf{E}^{\top}$ having rank $m$, with $m \le n$, and let $\{\mathbf{w}_{1}, \ldots, \mathbf{w}_{q-m}\}$ be a basis of $\mathcal{W}:=\mathcal{E}^{\perp}$. In addition, define matrix $E^{\top}:= [\mathbf{e}_{1}, \ldots, \mathbf{e}_{m}] \in \mathbb{R}^{q \times m}$, and matrix $W^{\top}:= [\mathbf{w}_{1}, \ldots, \mathbf{w}_{q-m}] \in \mathbb{R}^{q \times (q-m)}$, then for any $\vec{\beta}\in \mathbb{R}^{q}$ it holds
 \begin{enumerate}
 \item $\vec{\beta}=[E^{\top}\; W^{\top}] \cdot \mathbf{c}$ where $\mathbf{c} \in \mathbb{R}^{q}$ is a coefficient vector, and
 \item the matrix $[E^{\top}\; W^{\top}] \in \mathbb{R}^{q \times q}$ is invertible, that is, we have   
\begin{equation*}
    [E^{\top}\; W^{\top}]^{-1} = 
 \begin{bmatrix}[l]
   (E\,E^{\top})^{-1}\,E \\
   (W\,W^{\top})^{-1}\,W 
 \end{bmatrix}. 
 \end{equation*}
 \end{enumerate} 
\end{lemma}
\begin{proof}
  For a proof see~\citet[pp.~90-91]{mei:13}.
\end{proof}

Notice that $\mathcal{E}$ can be interpreted as indicating a vector subspace of balanced excesses. A pre-imputation will be mapped into its proper vector subspace of balanced excesses $\mathcal{E}$, i.e. the vector subspace induced by the pre-imputation. However, the corresponding vector of (un)balanced excesses generated by this pre-imputation is an element of this vector subspace of balanced excesses, if the pre-imputation is also a pre-kernel point. Hence, the vector of balanced excesses coincides with the vector of balanced maximum surpluses. This is a consequence of Lemma~\ref{lem:Etx_Pa} or see Proposition 8.4.1 in~\citet{mei:13}. Otherwise, this vector of unbalanced excesses will be mapped by the orthogonal projection $\mathbf{P}$ on $\mathcal{E}$. More information about the properties of this kind of vector subspace can be found in~\citet[pp.~87-113~and~138-168]{mei:13}.

\begin{proposition}[Positive General Linear Group]
  \label{prop:GLG}
Let $\{\mathbf{e}_{1}, \ldots, \mathbf{e}_{m}\}$ as well as $\{\mathbf{e}^{1}_{1}, \ldots, \mathbf{e}^{1}_{m}\}$ be two ordered bases of the subspace $\mathcal{E}$ derived from the payoff sets $[\vec{\gamma}]$ and $[\vec{\gamma}_{1}]$, respectively. In addition, define the associated basis matrices $E^{\top},E^{\top}_{1} \in \mathbb{R}^{q \times m}$ as in Lemma~\ref{lem:spV1}, then the unique transition matrix $X \in \mathbb{R}^{m^{2}}$ such that $E^{\top}_{1} = E^{\top} \,X$ is given, is an element of the positive general linear group, that is $X \in \text{GL}^{+}(m)$. 
\end{proposition}
\begin{proof}
  The proof can be found in~\citet[p.~101]{mei:13}.
\end{proof}
Proposition~\ref{prop:GLG} denotes two payoff sets $[\vec{\gamma}]$ and $[\vec{\gamma}_{1}]$ as equivalent, if there exists a transition matrix $X$ from the positive general linear group, that is $X \in \text{GL}^{+}(m)$, such that $E^{\top}_{1} = E^{\top} \,X$ is in force. Notice that the transition matrix $X$ must be unique (cf.~\citet[p. 102]{mei:13}). The underlying group action (cf.~\citet[Corollary 6.6.1]{mei:13}) can be interpreted that a bargaining situation is transformed into an equivalent bargaining situation. For a thorough discussion of a group action onto the set of all ordered bases, the interested reader should consult~\citet[Sect. 6.6]{mei:13}.

The vector space $\mathbb{R}^{q}$ is an orthogonal decomposition by the subspaces $\mathcal{E}$ and $\mathcal{N}_{\mathbf{E}}$. We denote in the sequel a basis of the orthogonal complement of space $\mathcal{E}$ by $\{\mathbf{w}_{1}, \ldots, \mathbf{w}_{q-m}\}$. This subspace of $\mathbb{R}^{q}$ is identified by $\mathcal{W}:= \mathcal{N}_{\mathbf{E}} = \mathcal{E}^{\perp}$. In addition, we have $\mathbf{P}\,\mathbf{w}_{k} = \mathbf{0}$ for all $k \in \{1,\ldots, q-m\}$. Thus, we can obtain the following corollary

\begin{corollary}[\citet{mei:13}]
  \label{cor:innp}
  If $\vec{\gamma}$ induces the matrices $\mathbf{E} \in \mathbb{R}^{n \times q},\mathbf{E}^{\top} \in \mathbb{R}^{q \times n}$ determined by formula~\eqref{eq:matE}, then with respect to the Euclidean inner product, getting 
 \begin{enumerate}
 \item $\mathbb{R}^{q} = \mathcal{E} \oplus \mathcal{W} = \mathcal{E} \oplus \mathcal{E}^{\perp}$.
 \end{enumerate}
\end{corollary}

A consequence of the orthogonal projection method presented by the next theorem and corollary is that every payoff vector belonging to the intersection of the minimum set of function $h_{\gamma}$ and its payoff equivalence class $[\vec{\gamma}]$ is a pre-kernel element. This due to $h_{\gamma} = h$ on $[\vec{\gamma}]$.

\begin{theorem}[Orthogonal Projection Method]
  \label{thm:pmt_prk}
  Let $\vec{\gamma}_{k} \in [\vec{\gamma}]$ for $k = 1,2,3$. If $\vec{\gamma}_{2} \in M(h_{\gamma}) $ and $\vec{\gamma}_{k} \notin M(h_{\gamma})$ for $k = 1,3 $, then $\vec{\zeta}_{2} = \vec{\xi}_{2} = \mathbf{0}$, and consequently $\vec{\gamma}_{2} \in \mathcal{P\text{\itshape r}K}(v)$.
\end{theorem}
\begin{proof}
  For a proof see~\citet[pp.~109-111]{mei:13}.
\end{proof}

\begin{corollary}[\citet{mei:13}]
    \label{cor:xi_prk_2}
 Let be $[\vec{\gamma}]$ an equivalence class of dimension $3 \le m \le n$, and $\mathbf{x} \in M(h_{\gamma}) \cap [\vec{\gamma}] $, then $\vec{\alpha} = \mathbf{P}\,\vec{\alpha}$, and consequently $\mathbf{x} \in \mathcal{P\text{\itshape r}K}(v)$.
\end{corollary}

\section{The Uniqueness of the Pre-Kernel}
\label{sec:siva}
To study the uniqueness of the pre-kernel solution of a related TU game derived from a pre-kernel element of a default game, we need to know: (1) if the linear mapping of a pre-kernel element into a specific vector subspace of balanced excesses $\mathcal{E}$ consists of a single point, and (2) that there can not exist any other non-transversal vector subspace of balanced excesses $\mathcal{E}_{1}$ in which a linear transformation of pre-kernel element can be mapped. (3) It must be shown that the pre-kernel coincides with the pre-nucleolus of the set of related games, otherwise, it is obvious that there must exist at least a second pre-kernel point, namely the pre-nucleolus.

For conducting this line of investigation some additional concepts are needed. In a first step we introduce the definition of a {\bfseries unanimity game}, which is indicated as: $\mathbf{u}_{T}(S):=1$, if $T \subseteq S$, otherwise $\mathbf{u}_{T}(S):=0$, whereas $T \subseteq N, T \neq \emptyset$. The collection of all unanimity games forms a {\bfseries unanimity/game basis}. A formula to express the coordinates of this basis is given by  
\begin{equation*}
  v = \sum_{\substack{T \subset N, \\ T \neq \emptyset}}\, \lambda^{v}_{T}\, \mathbf{u}_{T} \iff  \lambda^{v}_{T} = \sum_{\substack{S \subset T, \\ S \neq \emptyset}}\, (-1)^{t-s} \cdot v(S),
\end{equation*}
if $\langle N,v \rangle$, where $\arrowvert S \arrowvert = s $, and $\arrowvert T \arrowvert = t$. A coordinate $\lambda^{v}_{T}$ is said to be an unanimity coordinate of game $\langle N,v \rangle$, and vector $\lambda^{v}$ is called the unanimity coordinates of game $\langle N,v \rangle$. Notice that we assume here that the game is defined in $\mathbb{R}^{2^{n}-1}$ rather than $\mathbb{R}^{2^{n}}$, since we want to write for sake of convenience the {\bfseries game basis} in matrix form without a column and row of zeros. Thus we write
\begin{equation*}
  v = \sum_{\substack{T \subset N, \\ T \neq \emptyset}}\, \lambda^{v}_{T}\, \mathbf{u}_{T} = [\mathbf{u}_{\{1\}}, \ldots , \mathbf{u}_{\{N\}}] \, \lambda^{v} = \boldsymbol{\EuScript{U}}\; \lambda^{v}
\end{equation*}
where the unanimity basis $\boldsymbol{\EuScript{U}}$ is in $\mathbb{R}^{p^{\prime} \times p^{\prime}}$ with $p^{\prime}=2^{n}-1$. In addition, define the {\bfseries unity games (Dirac games)} $\mathbf{1}^{T}$ for all $T \subseteq N$ as: $\mathbf{1}^{T}(S):=1$, if $T=S$, otherwise $\mathbf{1}^{T}(S):=0$.  

In the next step, we select a payoff vector $\vec{\gamma}$, which also determines its payoff set $[\vec{\gamma}]$. With regard to Proposition~\ref{prop:eq_rel}, this vector induces in addition a set of lexicographically smallest most effective coalitions indicated by $\mathcal{S}(\vec{\gamma})$. Implying that we get the configuration $\vec{\alpha}$ by the $q$-coordinates $\alpha_{ij} := (v(S_{ij}) - v(S_{ji})) \in \mathbb{R} $ for all $i,j \in N, i < j $, and $\alpha_{0} := v(N)$. Furthermore, we can also define a set of vectors as the differences of unity games w.r.t. the set of lexicographically smallest most effective coalitions, which is given by 
\begin{equation}
  \label{eq:mat_V}
  \mathbf{v}_{ij} := \mathbf{1}^{S_{ij}} - \mathbf{1}^{S_{ji}} \quad\text{for}\; S_{ij},S_{ji} \in \mathcal{S}(\vec{\gamma}) \quad\text{and}\quad   \mathbf{v}_{0} := \mathbf{1}^{N},
\end{equation}
whereas $\mathbf{v}_{ij}, \mathbf{v}_{0} \in \mathbb{R}^{p^{\prime}}$ for all $i,j \in N, i < j$. With these column vectors, we can identify matrix $\boldsymbol{\mathcal{V}}:= [\mathbf{v}_{1,2}, \ldots ,\mathbf{v}_{n-1,n},\mathbf{v}_{0}] \in \mathbb{R}^{p^{\prime} \times q}$. Then we obtain $\vec{\alpha} = \boldsymbol{\mathcal{V}}^{\top}\, v$ with $v \in \mathbb{R}^{p^\prime{}}$ due to the removed empty set. Moreover, by the measure $y(S):= \sum_{k \in S}\,{y}_{k}$ for all $\emptyset \neq S \subseteq N$, we extend every payoff vector $\mathbf{y}$ to a vector $\overline{\mathbf{y}} \in \mathbb{R}^{p^{\prime}}$, and define the excess vector at $\mathbf{y}$ by $\overline{e}_{\mathbf{y}} := v - \overline{\mathbf{y}} \in \mathbb{R}^{p^{\prime}}$, then we get $\vec{\xi}_{\mathbf{y}} = \boldsymbol{\mathcal{V}}^{\top}\, \overline{e}_{\mathbf{y}}$. From matrix $\boldsymbol{\mathcal{V}}^{\top}$, we can also derive an orthogonal projection $\mathbf{P}_{\mathcal{V}}$ specified by $\boldsymbol{\mathcal{V}}^{\top}\,(\boldsymbol{\mathcal{V}}^{\top})^{\dagger} \in \mathbb{R}^{q \times q}$ such that $\mathbb{R}^{q} = \mathcal{V} \oplus \mathcal{V}^{\perp}$ is valid, i.e.~the rows of matrix $\boldsymbol{\mathcal{V}}^{\top}$ are a spanning system of the vector subspace $\mathcal{V} \subseteq \mathbb{R}^{q \times q}$, thus $\mathcal{V}:=span\{\mathbf{v}^{\top}_{1,2}, \ldots, \mathbf{v}^{\top}_{n-1,n},\mathbf{v}^{\top}_{0}\}$. Vector subspace $\mathcal{V}$ reflects the power of the set of lexicographically smallest most effective coalitions. In contrast, vector subspace $\mathcal{E}$ reflects the ascribed unbalancedness in the coalition power w.r.t. the bilateral bargaining situation attained at $\vec{\gamma}$ through $\mathcal{S}(\vec{\gamma})$. The next results show how these vector subspaces are intertwined.

\begin{lemma}[\citet{mei:13}]
  \label{lem:inl_vp}
  Let $\mathbf{E}^{\top} \in \mathbb{R}^{q \times n}$ be defined as in Equation~\eqref{eq:matE}, $\boldsymbol{\mathcal{V}}^{\top}\in \mathbb{R}^{q \times p^{\prime}}$ as by Equation~\eqref{eq:mat_V}, then there exists a matrix $\mathbf{Z}^{\top} \in \mathbb{R}^{p^{\prime} \times n}$ such that $\mathbf{E}^{\top} = \boldsymbol{\mathcal{V}}^{\top}\,\mathbf{Z}^{\top}$ if, and only if, $\mathcal{R}_{\mathbf{E}^{\top}} \subseteq \mathcal{R}_{\boldsymbol{\mathcal{V}}^{\top}}$, that is, $\mathcal{E} \subseteq \mathcal{V}$. 
\end{lemma}
\begin{proof}
  The proof is given in~\citet[p.~141]{mei:13}.
\end{proof}
Notice that the minimal rank of matrix $\boldsymbol{\mathcal{V}}^{\top}$ is bounded by $\mathbf{E}^{\top}$ which is equal to $m < n$ with the consequence that we get in this case $\mathcal{V} = \mathcal{E}$. However, the maximal rank is equal to $q$, and then $\mathcal{V} = \mathbb{R}^{q}$ (cf.~\citet[Corollary 7.4.1]{mei:13}).

\begin{lemma}[\citet{mei:13}]
  \label{lem:vsp_al}
Let $\vec{\alpha}, \vec{\xi} \in \mathbb{R}^q$ as in Equation~\eqref{eq:unb_exc}, then the following relations are satisfied on the vector space $\mathcal{V}$:
  \begin{enumerate}
  \item $\mathbf{P}_{\mathcal{V}}\,\vec{\alpha} = \vec{\alpha} \in \mathcal{V}$
  \item $\mathbf{P}_{\mathcal{V}}\,\vec{\xi} = \vec{\xi} \in \mathcal{V}$
  \item $\mathbf{P}_{\mathcal{V}}\,(\vec{\xi} - \vec{\alpha}) = (\vec{\xi} - \vec{\alpha}) \in \mathcal{V}$
  \item $\mathbf{P}_{\mathcal{V}}\, \mathbf{E}^{\top} = \mathbf{P}\,\mathbf{E}^{\top} = \mathbf{E}^{\top}$, hence $\mathcal{E} \subseteq \mathcal{V}$
  \item $\mathbf{E}\,\mathbf{P}_{\mathcal{V}} = \mathbf{E}\,\mathbf{P} = \mathbf{E}$, hence $\mathcal{R}_{\mathbf{E}} \subseteq \mathcal{V}$.
  \end{enumerate}
\end{lemma}
\begin{proof}
  For a proof see~\citet[p.~142]{mei:13}.
\end{proof}

It was worked out by~\citet[Sect. 7.6]{mei:13} that a pre-kernel element of a specific game can be replicated as a pre-kernel element of a related game whenever the non-empty interior property of the payoff set, in which the pre-kernel element of default game is located, is satisfied. In this case, a full dimensional ellipsoid can be inscribed from which some bounds can be specified within the game parameter can be varied without destroying the pre-kernel properties of the payoff vector of the default game. These bounds specify a redistribution of the bargaining power among coalitions while supporting the selected pre-imputation still as a pre-kernel point. Although the values of the maximum excesses have been changed by the parameter variation, the set of lexicographically smallest most significant coalitions remains unaffected. 

\begin{lemma}[\citet{mei:13}]
  \label{lem:repl_min}
If $\mathbf{x} \in M(h^{v}_{\gamma})$, then $\mathbf{x} \in M(h^{v^{\mu}}_{\gamma})$ for all $\mu \in \mathbb{R}$, where $v^{\mu} := \boldsymbol{\EuScript{U}}(\lambda^{v} + \mu \Delta)$ and $\mathbf{0} \neq \Delta \in \mathcal{N}_{\boldsymbol{\EuScript{W}}}=\{\Delta \in \mathbb{R}^{p^{\,\prime}} \;\arrowvert\; \boldsymbol{\EuScript{W}}\Delta = \mathbf{0}\}$, where $\boldsymbol{\EuScript{W}} := \boldsymbol{\mathcal{V}}^{\top}\, \boldsymbol{\EuScript{U}} \in \mathbb{R}^{q \times p^{\,\prime}}$.
\end{lemma}
\begin{proof}
Let $\mathbf{x}$ be a minimizer of function $h^{v}_{\gamma}$ under game $v$, then $\mathbf{x}$ remains a minimizer for a function $h^{v^{\mu}}_{\gamma}$ induced by game $v^{\mu}$ whenever $\mathbf{Q}\, \mathbf{x}= -2\, \mathbf{E}\,\vec{\alpha}= - \mathbf{a}$ remains valid. Since the payoff vector has induced the matrices $\mathbf{Q}, \mathbf{E}$ and matrix $\boldsymbol{\mathcal{V}}$ defined by $ [\mathbf{v}_{1,2}, \ldots ,\mathbf{v}_{n-1,n},\mathbf{v}_{0}]$, where the vectors are defined as by formula~\eqref{eq:mat_V}. We simply have to prove that the configuration $\vec{\alpha}$ remains invariant against an appropriate change in the game parameter. Observing that matrix $\boldsymbol{\EuScript{W}}$ has a rank equal to or smaller than $q = \binom{n}{2} +1$, say $m\le q$, then the null space of matrix $\boldsymbol{\EuScript{W}}$ has rank of $p^{\prime}-m$, thus $\mathcal{N}_{\boldsymbol{\EuScript{W}}}\neq \{\emptyset\}$. But then exists some $\mathbf{0} \neq \Delta \in \mathbb{R}^{p^{\prime}}$ s.t. $\Delta \in \mathcal{N}_{\boldsymbol{\EuScript{W}}}$ and $v^{\mu} = \boldsymbol{\EuScript{U}}(\lambda^{v} + \mu \Delta)$ for $\mu \in \mathbb{R}\backslash\{\mathbf{0}\}$, getting
\begin{equation*}
  \boldsymbol{\EuScript{W}}\, \, \lambda^{v^{\mu}} = \boldsymbol{\EuScript{W}}\,(\lambda^{v} + \mu \Delta)= \boldsymbol{\mathcal{V}}^{\top}\,(v + \mu v^{\Delta})= \boldsymbol{\mathcal{V}}^{\top}\,v = \vec{\alpha},
\end{equation*}
whereas $\boldsymbol{\EuScript{W}}\, \Delta = \boldsymbol{\mathcal{V}}^{\top}\,v^{\Delta} = \mathbf{0}$ with $v^{\Delta}:=\boldsymbol{\EuScript{U}}\,\Delta$. This argument proves that the configuration $\vec{\alpha}$ remains invariant against a change in the game parameter space by $v^{\Delta} \neq \mathbf{0}$. This implies that the payoff vector $\mathbf{x}$ is also a minimizer for function $h^{v^{\mu}}_{\gamma}$ under game $v^{\mu}$. 
\end{proof}

\begin{lemma}[\citet{mei:13}]
  \label{lem:repl_prk}
If $[\vec{\gamma}]$ has non-empty interior and $\mathbf{x} \in \mathcal{P\text{\itshape r}K}(v) \subset [\vec{\gamma}]$, then there exists some critical bounds given by 
\begin{equation}
  \label{eq:crit_bds}
  \delta^{\varepsilon}_{ij}(\mathbf{x})=  \frac{\pm\sqrt{\bar{c}}}{\Arrowvert \mathbf{E}^{\top} (\mathbf{1}_{j}-\mathbf{1}_{i}) \Arrowvert} \neq 0 \qquad \forall i,j \in N, i \neq j,
\end{equation}
with $\bar{c}>0$ and $\Arrowvert \mathbf{E}^{\top} (\mathbf{1}_{j}-\mathbf{1}_{i}) \Arrowvert > 0$.
\end{lemma}
\begin{proof}
Define a set $\varepsilon := \{\mathbf{y}\, \arrowvert h^{v}_{\gamma}(\mathbf{y}) \le \bar{c} \} \subset [\vec{\gamma}]$, whereas $h^{v}_{\gamma}(\mathbf{y})=(1/2) \cdot \langle\; \mathbf{y},\mathbf{Q} \,\mathbf{y} \;\rangle + \langle\; \mathbf{y}, \mathbf{a} \;\rangle + \mathbf{\alpha}$. By assumption the payoff set $[\vec{\gamma}]$ has non-empty interior, we can say that $\varepsilon$ is the ellipsoid of maximum volume obtained by equation~\eqref{eq:objf2} that lies inside of the convex payoff set $[\vec{\gamma}]$. This ellipsoid must have a strictly positive volume, since the payoff equivalence class $[\vec{\gamma}]$ has non-empty interior, hence we conclude that $\bar{c}>0$. Of course, the set $\varepsilon$ is a convex subset of the convex set $[\vec{\gamma}]$, therefore $h^{v}=h^{v}_{\gamma}$ on $\varepsilon$. Moreover, the solution set $M(h^{v}_{\gamma})$ is a subset of the ellipsoid $\varepsilon$, which is the smallest non-empty ellipsoid of the form~\eqref{eq:objf2}, i.e., it is its center in view of Theorem~\ref{thm:pmt_prk}. By our supposition $\mathcal{P\text{\itshape r}K}(v) \subset [\vec{\gamma}]$, we conclude that $M(h^{v})=M(h^{v}_{\gamma})= \mathcal{P\text{\itshape r}K}(v)$ must be satisfied. In the next step similar to~\citet{MPSh:79}, we define some critical numbers $\delta^{\varepsilon}_{ij}(\mathbf{x}) \in \mathbb{R}$ s.t. 
\begin{equation}
  \label{eq:crit_num}
  \delta^{\varepsilon}_{ij}(\mathbf{x}):=\max\,\{\delta \in \mathbb{R}\,\arrowvert\, \mathbf{x}^{\;i,j,\delta}= \mathbf{x} - \delta\,\mathbf{1}_{i} + \delta\,\mathbf{1}_{j} \in \varepsilon  \} \qquad \forall\, i,j \in N, i\neq j.
\end{equation}
That is, the number $\delta^{\varepsilon}_{ij}(\mathbf{x})$ is the maximum amount that can be transferred from $i$ to $j$ while remaining in the ellipsoid $\varepsilon$. This number is well defined for convex sets having non-empty interior. 

In addition, observe that $\mathbf{x}^{\;i,j,\delta^{\varepsilon}}= \mathbf{x} - \delta^{\varepsilon}_{ij}(\mathbf{x})\,\mathbf{1}_{i} + \delta^{\varepsilon}_{ij}(\mathbf{x})\,\mathbf{1}_{j}$ is an unique boundary point of the ellipsoid $\varepsilon$ of type~\eqref{eq:objf2} with maximum volume. Having specified by the point $\mathbf{x}^{\;i,j,\delta^{\varepsilon}}$ a boundary point, getting
\begin{equation*}
 \begin{split}
  & h^{v}(\mathbf{x}^{\;i,j,\delta^{\varepsilon}}) =h^{v}_{\gamma}(\mathbf{x}^{\;i,j,\delta^{\varepsilon}})=\bar{c} > 0 \Longleftrightarrow\\
  & \Arrowvert \mathbf{E}^{\top}\; \mathbf{x}^{\;i,j,\delta^{\varepsilon}} + \vec{\alpha}\Arrowvert^2 =\bar{c}  \Longleftrightarrow \Arrowvert \mathbf{E}^{\top}\; \mathbf{x} + \vec{\alpha} + \delta^{\varepsilon}_{ij}(\mathbf{x}) \, \mathbf{E}^{\top} (\mathbf{1}_{j}-\mathbf{1}_{i})\Arrowvert^2 =\bar{c} \Longleftrightarrow \\
& \Arrowvert \mathbf{E}^{\top}\; \mathbf{x} + \vec{\alpha} \Arrowvert^2 + 2 \cdot \delta^{\varepsilon}_{ij}(\mathbf{x}) \, \langle\,\mathbf{E}^{\top}\; \mathbf{x} + \vec{\alpha}, \mathbf{E}^{\top} (\mathbf{1}_{j}-\mathbf{1}_{i})\,\rangle + (\delta^{\varepsilon}_{ij}(\mathbf{x}))^{2}\,\Arrowvert  \mathbf{E}^{\top} (\mathbf{1}_{j}-\mathbf{1}_{i})\Arrowvert^2 =\bar{c} \Longleftrightarrow \\ &  (\delta^{\varepsilon}_{ij}(\mathbf{x}))^{2}\,\Arrowvert  \mathbf{E}^{\top} (\mathbf{1}_{j}-\mathbf{1}_{i})\Arrowvert^2  =\bar{c} \qquad \forall\, i,j \in N, i\neq j.
 \end{split}
\end{equation*}

The last conclusion follows, since by assumption we have $\mathbf{x} \in \mathcal{P\text{\itshape r}K}(v)$, which is equivalent to $h^{v}(\mathbf{x})=h^{v}_{\gamma}(\mathbf{x})= 0$, and therefore we obtain $\mathbf{E}^{\top}\; \mathbf{x} + \vec{\alpha}=\mathbf{0}$. In addition, the volume of the ellipsoid $\varepsilon$ is strictly positive such that $\bar{c}>0$, this result implies that $(\delta^{\varepsilon}_{ij}(\mathbf{x}))^{2}$ as well as $\Arrowvert  \mathbf{E}^{\top} (\mathbf{1}_{j}-\mathbf{1}_{i})\Arrowvert^{2}$ must also be strictly positive. Therefore, we get finally~\eqref{eq:crit_bds}.
\end{proof}

\begin{theorem}[\citet{mei:13}]
  \label{thm:repl_prk}
If $[\vec{\gamma}]$ has non-empty interior and $\mathbf{x} \in \mathcal{P\text{\itshape r}K}(v) \subset [\vec{\gamma}]$, then $\mathbf{x} \in \mathcal{P\text{\itshape r}K}(v^{\mu})$ for all $\mu\cdot v^{\Delta} \in [-\mathsf{C},\mathsf{C}]^{p^{\prime}}$, where $v^{\mu} = v + \mu\cdot v^{\Delta} \in \mathbb{R}^{p^{\prime}}$, $\mu \in \mathbb{R}$
\begin{equation}
 \label{eq:crt_bds}
  \mathsf{C} :=\min_{i,j \in N, i \neq j}\bigg\{\bigg\arrowvert \frac{\pm\sqrt{\bar{c}}}{\Arrowvert \mathbf{E}^{\top} (\mathbf{1}_{j}-\mathbf{1}_{i}) \Arrowvert}\bigg\arrowvert\bigg\},
\end{equation}
 and $\mathbf{0} \neq \Delta \in \mathcal{N}_{\boldsymbol{\EuScript{W}}}=\{\Delta \in \mathbb{R}^{p^{\,\prime}} \;\arrowvert\; \boldsymbol{\EuScript{W}}\Delta = \mathbf{0}\}$ with matrix $\boldsymbol{\EuScript{W}} := \boldsymbol{\mathcal{V}}^{\top}\, \boldsymbol{\EuScript{U}}$.
\end{theorem}
\begin{proof}
By Lemma~\ref{lem:repl_prk} $\mathbf{x}^{\;i,j,\delta^{\varepsilon}} \in \varepsilon \subset [\vec{\gamma}]$, is an unique boundary point of the ellipsoid $\varepsilon$ of type~\eqref{eq:objf2} with maximum volume. We conclude that either (1) $s_{ij}(\mathbf{x}^{\;i,j,\delta^{\varepsilon}}) = s_{ij}(\mathbf{x}) + \delta^{\varepsilon}_{ij}(\mathbf{x})$ if $ S \in \mathcal{G}_{ij}$, or (2) $s_{ji}(\mathbf{x}^{\;i,j,\delta^{\varepsilon}}) = s_{ji}(\mathbf{x}) - \delta^{\varepsilon}_{ij}(\mathbf{x})$ if $S \in \mathcal{G}_{ji}$, or otherwise (3) $s_{ij}(\mathbf{x}^{\;i,j,\delta^{\varepsilon}}) = s_{ij}(\mathbf{x})$ is satisfied. Moreover, let $v,v^{\mu},v^{\Delta} \in \mathbb{R}^{p^{\prime}}$ and recall that $v^{\mu} = \boldsymbol{\EuScript{U}}(\lambda^{v} + \mu \Delta)$ with $\mathbf{0} \neq \Delta \in \mathcal{N}_{\boldsymbol{\EuScript{W}}}$. Then it holds $v^{\mu}(S) = v(S) + \mu\cdot v^{\Delta}(S)$ for all $S \in 2^{n}\backslash\{\emptyset\} $. In the next step, extend the pre-kernel element $\mathbf{x}$ to a vector $\overline{\mathbf{x}}$ by the measure $x(S):=\sum_{k \in S}\,x_{k}$ for all $S \in 2^{n}\backslash\{\emptyset\} $, then define the excess vector by $\overline{e}:=v-\overline{\mathbf{x}}$. Due to these definitions, we obtain for $\vec{\xi}^{\,v^{\mu}}$ at $\mathbf{x}$:
\begin{equation*}
  \vec{\xi}^{\,v^{\mu}}=\boldsymbol{\mathcal{V}}^{\top}\,\overline{e}^{\mu}= \boldsymbol{\mathcal{V}}^{\top}\,(v^{\mu}-\overline{\mathbf{x}})=\boldsymbol{\mathcal{V}}^{\top}\,( v -\overline{\mathbf{x}} + \mu \cdot v^{\Delta}) = \boldsymbol{\mathcal{V}}^{\top}\,(v-\overline{\mathbf{x}}) = \boldsymbol{\mathcal{V}}^{\top}\,\overline{e} = \vec{\xi} = \mathbf{0}.
\end{equation*}
By Lemma~\ref{lem:repl_min}, the system of excesses remains balanced for all $\mu \in \mathbb{R}$. However, the system of maximum surpluses remains invariant on a hypercube specified by the critical values of the ellipsoid $\varepsilon$. Thus, for appropriate values of $\mu$ the expression $\mu\cdot v^{\Delta}(S)$ belongs to the non-empty interval $[-\mathsf{C},\mathsf{C}]$ for $S \in 2^{n}\backslash\{\emptyset\}$. This interval specifies the range in which the game parameter can vary without having any impact on the set of most effective coalition given by $\mathcal{S}(\mathbf{x})$. Thus, the coalitions $\mathcal{S}(\mathbf{x})$ still have  maximum surpluses for games defined by $v^{\mu} = \boldsymbol{\EuScript{U}}(\lambda^{v} + \mu \Delta)$ for all $\mu \boldsymbol{\EuScript{U}}\,\Delta = \mu \cdot v^{\Delta} \in [-\mathsf{C},\mathsf{C}]^{p^{\prime}}$. Hence the pre-kernel solution $\mathbf{x}$ is invariant against a change in the hypercube $[-\mathsf{C},\mathsf{C}]^{p^{\prime}}$. The conclusion follows.
\end{proof}

\citet[Sec. 7.6]{mei:13} has shown by some examples that the specified bounds by Theorem~\ref{thm:repl_prk} are not tight, in the sense that pre-kernel points belonging to the relative interior of a payoff set can also be the object of a replication. However, pre-kernel elements which are located on the relative boundary of a payoff set are probably not replicable. Therefore, there must exist a more general rule to reproduce a pre-kernel element for a related game $v^{\mu}$.   

In the course of our discussion, we establish that the single pre-kernel element of a default game which is an interior point of a payoff set is also the singleton pre-kernel of the derived related games. In a first step, we show that there exists an unique linear transformation of the pre-kernel point of a related game into the vector subspace of balanced excesses $\mathcal{E}$. This means, there is no other pre-kernel element in a payoff equivalence class that belongs to the same set of ordered bases, i.e.~reflecting an equivalent bargaining situation with a division of the proceeds of mutual cooperation in accordance with the pre-kernel solution. Secondly, we prove that there can not exist any other vector subspace of balanced excesses $\mathcal{E}_{1}$ non-transversal to $\mathcal{E}$ in which a pre-kernel vector can be mapped by a linear transformation. That is, there exists no other non-equivalent payoff set in which an other pre-kernel point can be located.   

\begin{lemma}[\citet{mei:13}]
  \label{lem:pivET}
 Let $\vec{\gamma}$ induces matrix $\mathbf{E}$, then
\begin{equation*}
  (\mathbf{E}^{\top})^{\dagger} = 2\cdot\mathbf{Q}^{\dagger}\mathbf{E} \in \mathbb{R}^{n \times q}.
\end{equation*}
\end{lemma}
\begin{proof}
  Remind from Lemma~\ref{lem:Etx_Pa} that $\mathbf{P} = 2 \cdot \mathbf{E}^{\top}\, \mathbf{Q}^{\dagger} \mathbf{E}$ holds. In addition, note that we have the following relation $\mathbf{Q}^{\dagger}\mathbf{Q} = (\mathbf{E}^{\top})^{\dagger}\,\mathbf{E}^{\top}$ which is an orthogonal projection onto $\mathcal{R}_{\mathbf{E}}$. Then attaining
  \begin{equation*}
   \begin{split}
    2\cdot\mathbf{Q}^{\dagger}\mathbf{E} & = 2\cdot\mathbf{Q}^{\dagger}\mathbf{Q}\mathbf{Q}^{\dagger}\mathbf{E}=2\cdot(\mathbf{E}^{\top})^{\dagger}\mathbf{E}^{\top}\mathbf{Q}^{\dagger}\mathbf{E} \\
    & =(\mathbf{E}^{\top})^{\dagger}(2\cdot\mathbf{E}^{\top}\mathbf{Q}^{\dagger}\mathbf{E})= (\mathbf{E}^{\top})^{\dagger}\mathbf{P} = (\mathbf{E}^{\top})^{\dagger}. 
  \end{split}
  \end{equation*}
The last equality follows from Lemma~\ref{lem:vsp_al}. This argument terminates the proof.
\end{proof}

Notice that in the sequel $\text{SO}(n)$ denotes the special orthogonal group, whereas $\text{GL}^{+}(n)$ denotes the positive general linear group (cf.~\citet[pp.~99-109]{mei:13}).

\begin{proposition} 
  \label{prop:uniqpk01}
Let $E^{\top}_{1} = E^{\top}\,X$ with $X \in \text{SO}(n)$, that is $[\vec{\gamma}] \thicksim [\vec{\gamma}_{1}]$. In addition, assume that the payoff equivalence class $[\vec{\gamma}]$ induced from TU game $\langle\, N, v\, \rangle$ has non-empty interior such that $\{\mathbf{x}\} = \mathcal{P \text{\itshape r}K}(v) \subset [\vec{\gamma}]$ is satisfied, then there exists no other pre-kernel element in payoff equivalence class $[\vec{\gamma}_{1}]$ for a related TU game $\langle\, N, v^{\mu}\, \rangle$, where $v^{\mu} = v + \mu\cdot v^{\Delta} \in \mathbb{R}^{p^{\prime}}$, as defined by Lemma~\ref{lem:repl_min}.
\end{proposition}
\begin{proof}
  By the way of contradiction suppose that $\mathbf{x},\mathbf{y} \in \mathcal{P\text{\itshape r}K}(v^{\mu})$ with $\mathbf{y} \in [\vec{\gamma}_{1}]$ is valid. Then we get 
  \begin{equation*}
        h^{v^{\mu}}(\mathbf{x}) = h_{\gamma}^{v^{\mu}}(\mathbf{x}) = \Arrowvert \mathbf{E}^{\top}\; \mathbf{x} + \vec{\alpha} \Arrowvert^2 = 0 \quad \text{and}\quad h^{v^{\mu}}(\mathbf{y}) = h_{\gamma_{1}}^{v^{\mu}}(\mathbf{y}) = \Arrowvert \mathbf{E}_{1}^{\top}\; \mathbf{y} + \vec{\alpha}_{1} \Arrowvert^2 = 0, 
  \end{equation*}
implying that
\begin{equation}
 \label{eq:mapV}
  \mathbf{P}\,\vec{\alpha} = \vec{\alpha} \in \mathcal{E} \qquad\text{and}\qquad \mathbf{P}\,\vec{\alpha}_{1} = \vec{\alpha}_{1} \in \mathcal{E}.
\end{equation}
Moreover, we have $E^{\top}_{1} = E^{\top}\,X$ with $X \in \text{SO}(n)$, then $\mathcal{E} \subseteq \mathcal{V} \cap \mathcal{V}_{1}$ in accordance with Lemma 7.4.1 by~\citet{mei:13}. Now assume that $\vec{\alpha}_{1} = \boldsymbol{\mathcal{V}}^{\top}_{1}\,v^{\mu}$ holds with $\mathcal{V}_{1} \subseteq \mathcal{V}$. The latter supposition implies $\boldsymbol{\mathcal{V}}^{\top}_{1} = \mathbf{P}_{\mathcal{V}}\,\boldsymbol{\mathcal{V}}^{\top}_{1}$, since for every $\vec{\beta} \in \mathcal{V}$ we get $\vec{\beta} = \mathbf{P}_{\mathcal{V}}\,\vec{\beta}$ (cf.~Remark 6.5.1 ~\citet{mei:13}). According to $\mathcal{V}_{1} \subseteq \mathcal{V}$ it also holds $\mathcal{N}_{\boldsymbol{\EuScript{W}_{1}}} \supseteq \mathcal{N}_{\boldsymbol{\EuScript{W}}}$. Our hypothesis $\mathbf{y} \in \mathcal{P\text{\itshape r}K}(v^{\mu})$ implies
\begin{equation*}
  \mathbf{0} = \mathbf{E}_{1}^{\top}\; \mathbf{y} + \vec{\alpha}_{1} = \boldsymbol{\mathcal{V}}^{\top}_{1}\,\mathbf{Z}^{\top}\,\mathbf{y} + \boldsymbol{\mathcal{V}}^{\top}_{1}\,v^{\mu} = \boldsymbol{\mathcal{V}}^{\top}_{1}\,\mathbf{Z}^{\top}\; \mathbf{y} + \boldsymbol{\mathcal{V}}^{\top}_{1}\,(v + \mu\cdot v^{\Delta}) = \boldsymbol{\mathcal{V}}^{\top}_{1}\,(v - \overline{\mathbf{y}}),
\end{equation*}
whereas the vector of measures $\overline{\mathbf{y}}$ is expressed by $\overline{\mathbf{y}} = -\mathbf{Z}^{\top}\,\mathbf{y}$ (cf.~\citet[p. 141]{mei:13}). The result $\boldsymbol{\mathcal{V}}^{\top}_{1}\,(v - \overline{\mathbf{y}}) = \mathbf{0}$ yields to $\mathbf{y} \in \mathcal{P\text{\itshape r}K}(v)$, which is a contradiction. Therefore, we conclude that $\mathcal{V} \subset \mathcal{V}_{1}$ must be satisfied. 

In addition, from $\vec{\alpha}_{1} = \boldsymbol{\mathcal{V}}^{\top}_{1}\,v^{\mu}$ we attain $\mathbf{P}_{\mathcal{V}}\,\vec{\alpha}_{1} = \boldsymbol{\mathcal{V}}^{\top}\,(\boldsymbol{\mathcal{V}}^{\top})^{\dagger}\,\boldsymbol{\mathcal{V}}^{\top}_{1}\,v^{\mu} \neq \boldsymbol{\mathcal{V}}^{\top}_{1}\,v^{\mu} = \vec{\alpha}_{1}$ in accordance with $\mathbf{P}_{\mathcal{V}}\,\boldsymbol{\mathcal{V}}^{\top}_{1} \neq \boldsymbol{\mathcal{V}}^{\top}_{1}$, in fact, it holds $\mathcal{V} \subset \mathcal{V}_{1}$. Thus, we have $\mathbf{P}_{\mathcal{V}}\,\vec{\alpha}_{1} \notin \mathcal{V}$ contradicting that $\mathbf{P}_{\mathcal{V}}\,\vec{\alpha}_{1} = \vec{\alpha}_{1} \in  \mathcal{E} \subseteq \mathcal{V} \subset \mathcal{V}_{1}$ holds true. From this, we conclude that $\vec{\alpha}_{1} = \boldsymbol{\mathcal{V}}^{\top}\,v^{\mu}$ must be in force. 

Furthermore, from~\eqref{eq:mapV} we have 
\begin{equation*}
  \mathbf{P}\,\vec{\alpha} - \vec{\alpha} = \mathbf{P}\,\vec{\alpha}_{1} - \vec{\alpha}_{1} = \mathbf{0} \in \mathcal{E} \Longleftrightarrow \mathbf{P}\,(\vec{\alpha} - \vec{\alpha}_{1}) = (\vec{\alpha} - \vec{\alpha}_{1}) \in \mathcal{E}.
\end{equation*}
Therefore, obtaining the equivalent expression
\begin{equation*}
  \mathbf{E}^{\top}\;(X\,\mathbf{y} - \mathbf{x}) = (\vec{\alpha} - \vec{\alpha}_{1}) = \boldsymbol{\mathcal{V}}^{\top}\,v - \boldsymbol{\mathcal{V}}^{\top}\,(v + \mu\cdot v^{\Delta}) = \mathbf{0},
\end{equation*}
then $\mathbf{x} = X\,\mathbf{y}$, since matrix $\mathbf{E}^{\top}$ has full rank due to $\{\mathbf{x}\} = \mathcal{P\text{\itshape r}K}(v)$.
Furthermore, notice that
\begin{equation*}
    \langle\,\mathbf{x},\mathbf{y}\,\rangle = \langle\,(\mathbf{E}^{\top})^{\dagger}\,\vec{\alpha},(\mathbf{E}_{1}^{\top})^{\dagger}\,\vec{\alpha}_{1}\,\rangle = \langle\,(\mathbf{E}^{\top})^{\dagger}\,\vec{\alpha},X^{-1}\,(\mathbf{E}^{\top})^{\dagger}\,\vec{\alpha}\,\rangle = \langle\,2\,\mathbf{Q}^{\dagger}\,\mathbf{E}\,\vec{\alpha},2\,X^{-1}\,\mathbf{Q}^{\dagger}\,\mathbf{E}\,\vec{\alpha}\,\rangle \neq \mathbf{0}
\end{equation*}
Matrix $\mathbf{E}^{\top}$ has full rank, and $\mathbf{Q}$ is symmetric and positive definite, hence $\mathbf{Q}^{\dagger} = \mathbf{Q}^{-1}$, and the above expression can equivalently be written as
\begin{equation}
 \label{eq:simmat}
  \begin{split}
    \langle\,\mathbf{Q}^{\dagger}\,\mathbf{a},X^{-1}\,\mathbf{Q}^{\dagger}\,\mathbf{a} \,\rangle & = \langle\,\mathbf{Q}^{-1}\,\mathbf{a},X^{-1}\,\mathbf{Q}^{-1}\,\mathbf{a} \,\rangle = \langle\,\mathbf{a},\mathbf{Q}\,X^{-1}\mathbf{Q}^{-1}\,\mathbf{a} \,\rangle \\
     & = \langle\,\mathbf{a},X_{1}\mathbf{a} \,\rangle = \langle\,\mathbf{a},\mathbf{a}_{1} \,\rangle \neq \mathbf{0}, 
  \end{split}
\end{equation}
while using $\mathbf{a} = 2\,\mathbf{E}\,\vec{\alpha}$ from Proposition~\ref{prop:eqrep}, and with similar matrix $X_{1} = \mathbf{Q}\,X^{-1}\mathbf{Q}^{-1}$ as well as $\mathbf{a}_{1} = X_{1}\,\mathbf{a}$. According to $\mathbf{E}^{\top}_{1} = \mathbf{E}^{\top}\,X$ with $X \in \text{SO}(n)$, we can write $X = \mathbf{Q}^{-1}(2\,\mathbf{E}\, \mathbf{E}_{1}^{\top})$. But then 
 \begin{equation*}
  X_{1} = \mathbf{Q}\,X^{-1}\mathbf{Q}^{-1} = \mathbf{Q}\, (2\,\mathbf{E}\,\mathbf{E}_{1}^{\top})^{-1}. 
\end{equation*}
Since we have $X \in \text{SO}(n)$, it holds $X^{-1} = X^{\top}$ implying that
\begin{equation*}
  X_{1}^{\top} = X^{-1} = (2\,\mathbf{E}\, \mathbf{E}_{1}^{\top})^{-1}\,\mathbf{Q} = (2\,\mathbf{E}\, \mathbf{E}_{1}^{\top}) \,\mathbf{Q}^{-1} = X^{\top} = X_{1}^{-1}, 
\end{equation*}
which induces $X=\mathbf{Q}^{-1}\,(2\,\mathbf{E}\, \mathbf{E}_{1}^{\top}) = \mathbf{Q}\, (2\,\mathbf{E}\,\mathbf{E}_{1}^{\top})^{-1} = X_{1}$. Now, observe
\begin{equation*}
  \begin{split}
  X_{1} & = \mathbf{Q}\,X^{-1}\mathbf{Q}^{-1} = \mathbf{Q}\,X^{\top}\mathbf{Q}^{-1} = \mathbf{Q}\,(2\,\mathbf{E}\,\mathbf{E}_{1}^{\top})\,\mathbf{Q}^{-1}\,\mathbf{Q}^{-1} \\
       & = \mathbf{Q}\,(2\,\mathbf{E}\,\mathbf{E}^{\top}\,X)\,\mathbf{Q}^{-2} = \mathbf{Q}^{2}\,X\,\mathbf{Q}^{-2},
  \end{split}
\end{equation*}
hence, we can conclude that $X = \mathbf{I}$ implying $X_{1} = \mathbf{I}$ as well. We infer that $\mathbf{x} = \mathbf{y}$ contradicting the assumption $\mathbf{x} \neq \mathbf{y}$ due to $\mathbf{x} \in [\vec{\gamma}]$, and $\mathbf{y} \in [\vec{\gamma}_{1}]$. With this argument we are done.
\end{proof}

\begin{proposition}
  \label{prop:uniqpk01b}
Impose the same conditions as under Proposition~\ref{prop:uniqpk01} with the exception that $X \in \text{GL}^{+}(n)$, then there exists no other pre-kernel element in payoff equivalence class $[\vec{\gamma}_{1}]$ for a related TU game $\langle\, N, v^{\mu}\, \rangle$.
\end{proposition}
\begin{proof}
 By the proof of Proposition~\ref{prop:uniqpk01} the system of linear equations $\mathbf{E}^{\top}\,(X\,\mathbf{y} - \mathbf{x}) = \mathbf{0}$ is consistent, then we get $\mathbf{x} = X\,\mathbf{y}$ by the full rank of matrix $\mathbf{E}^{\top}$. By Equation~\ref{eq:simmat} we obtain similar matrix $X_{1} = \mathbf{Q}\,X^{-1}\mathbf{Q}^{-1}$, hence the matrix $X_{1}$ is in the same orbit (conjugacy class) as matrix $X^{-1}$, this implies that $E^{\top} = E^{\top}_{1} \,X^{-1} = E^{\top}_{1} \,X_{1}$ must be in force. But then $E^{\top} = E^{\top}\,X \,X_{1}$, which requires that $X \,X_{1} = \mathbf{I}$ must be satisfied in view of the uniqueness of the transition matrix $X \in \text{GL}^{+}(m)$ (cf.~\citet[p. 102]{mei:13}). In addition, we have $\mathbf{a}_{1} = X_{1}\,\mathbf{a}$ as well as $\mathbf{a}_{1} = 2\,\mathbf{E}_{1}\,\vec{\alpha} = X\,\mathbf{a}$. Therefore, we obtain $X\,\mathbf{a}_{1} = \mathbf{a} = X^{2}\,\mathbf{a}$. From this we draw the conclusion in connection with the uniqueness of the transition matrix $X$ that $X = \mathbf{I}$ is valid. Hence, $\mathbf{x} = \mathbf{y}$ as required.
\end{proof}

\begin{proposition}
 \label{prop:uniqpk02}
  Assume $[\vec{\gamma}] \nsim [\vec{\gamma}_{1}]$, and that the payoff equivalence class $[\vec{\gamma}]$ induced from TU game $\langle\, N, v\, \rangle$ has non-empty interior such that $\{\mathbf{x}\} = \mathcal{P \text{\itshape r}K}(v) \subset [\vec{\gamma}]$ is satisfied, then there exists no other pre-kernel element in payoff equivalence class $[\vec{\gamma}_{1}]$ for a related TU game $\langle\, N, v^{\mu}\, \rangle$, where $v^{\mu} = v + \mu\cdot v^{\Delta} \in \mathbb{R}^{p^{\prime}}$, as defined by Lemma~\ref{lem:repl_min}.
\end{proposition}
\begin{proof}
  We have to establish that there is no other element $\mathbf{y} \in \mathcal{P\text{\itshape r}K}(v^{\mu})$ such that $\mathbf{y} \in [\vec{\gamma}_{1}]$ is valid, whereas $\mathbf{y} \notin \mathcal{P\text{\itshape r}K}(v)$ in accordance with the uniqueness of the pre-kernel for game $v$. In view of Theorem~\ref{thm:repl_prk} the pre-kernel $\{\mathbf{x}\} = \mathcal{P\text{\itshape r}K}(v)$ of game $\langle\, N, v\, \rangle$ is also a pre-kernel element of the related game $\langle\, N, v^{\mu}\, \rangle$, i.e. $\mathbf{x} \in  \mathcal{P\text{\itshape r}K}(v^{\mu})$ with $\mathbf{x} \in [\vec{\gamma}]$ due to Corollary~\ref{cor:innp}. 
  
  Extend the payoff element $\mathbf{y}$ to a vector $\overline{\mathbf{y}}$ by the measure $y(S):=\sum_{k \in S}\,y_{k}$ for all $S \in 2^{n}\backslash\{\emptyset\} $, then define the excess vector by $\overline{e}^{\mu}:=v^{\mu}-\overline{\mathbf{y}}$. Moreover, compute the vector of (un)balanced excesses $\vec{\xi}^{\,v^{\mu}}$ at $\mathbf{y}$ for game $v^{\mu}$ by $\boldsymbol{\mathcal{V}}_{1}^{\top}\,\overline{e}^{\mu}$. This vector is also the vector of (un)balanced maximum surpluses, since $\mathbf{y} \in [\vec{\gamma}_{1}]$, and therefore $h^{\,v^{\mu}}=h^{\,v^{\mu}}_{\gamma_{1}}$ on $[\vec{\gamma}_{1}]$ in view of Lemma 6.2.2 by~\citet{mei:13}. Notice that in order to have a pre-kernel element at $\mathbf{y}$ for the related game $v^{\mu}$ it must hold $\vec{\xi}^{\,v^{\mu}} = \mathbf{0}$. In addition, by hypothesis $[\vec{\gamma}] \nsim [\vec{\gamma}_{1}]$, it must hold $\mathbf{E}^{\top} = \boldsymbol{\mathcal{V}}^{\top}\,\mathbf{Z}^{\top}$ and $\mathbf{E}_{1}^{\top} = \boldsymbol{\mathcal{V}}_{1}^{\top}\,\mathbf{Z}^{\top}$ in view of Lemma~\ref{lem:inl_vp}, thus $E^{\top}_{1} \not= E^{\top}\,X$ for all $X \in \text{GL}^{+}(n)$. This implies that we derive the corresponding matrices $\boldsymbol{\EuScript{W}} := \boldsymbol{\mathcal{V}}^{\top}\, \boldsymbol{\EuScript{U}}$ and $\boldsymbol{\EuScript{W}}_{1} := \boldsymbol{\mathcal{V}}_{1}^{\top}\, \boldsymbol{\EuScript{U}}$, respectively. 

 We have to consider two cases, namely $\Delta \in \mathcal{N}_{\boldsymbol{\EuScript{W}}} \cap \mathcal{N}_{\boldsymbol{\EuScript{W}_{1}}}$ and $\Delta \in \mathcal{N}_{\boldsymbol{\EuScript{W}}} \backslash \mathcal{N}_{\boldsymbol{\EuScript{W}_{1}}}$.
\begin{enumerate}
\item Suppose $\Delta \in \mathcal{N}_{\boldsymbol{\EuScript{W}}} \cap \mathcal{N}_{\boldsymbol{\EuScript{W}_{1}}}$, then we get
\begin{equation*}
  \vec{\xi}^{\,v^{\mu}}=\boldsymbol{\mathcal{V}}_{1}^{\top}\,\overline{e}^{\mu}= \boldsymbol{\mathcal{V}}_{1}^{\top}\,(v^{\mu}-\overline{\mathbf{y}})=\boldsymbol{\mathcal{V}}_{1}^{\top}\,( v -\overline{\mathbf{y}} + \mu \cdot v^{\Delta}) = \boldsymbol{\mathcal{V}}_{1}^{\top}\,(v-\overline{\mathbf{y}}) = \boldsymbol{\mathcal{V}}_{1}^{\top}\,\overline{e} = \vec{\xi}^{\,v} \not= \mathbf{0}.
\end{equation*}
Observe that $\vec{\xi}^{\,v} = \boldsymbol{\mathcal{V}}_{1}^{\top}\,(v-\overline{\mathbf{y}}) \not= \mathbf{0}$, since vector $\mathbf{y} \in [\vec{\gamma}_{1}]$ is not a pre-kernel element of game $v$. 
\item Now suppose $\Delta \in \mathcal{N}_{\boldsymbol{\EuScript{W}}} \backslash \mathcal{N}_{\boldsymbol{\EuScript{W}_{1}}}$, then
\begin{equation*}
  \vec{\xi}^{\,v^{\mu}}=\boldsymbol{\mathcal{V}}_{1}^{\top}\,\overline{e}^{\mu}= \boldsymbol{\mathcal{V}}_{1}^{\top}\,(v^{\mu}-\overline{\mathbf{y}})=\boldsymbol{\mathcal{V}}_{1}^{\top}\,( v -\overline{\mathbf{y}} + \mu \cdot v^{\Delta}) = \boldsymbol{\mathcal{V}}_{1}^{\top}\,\overline{e} + \mu \cdot \boldsymbol{\mathcal{V}}_{1}^{\top}\,v^{\Delta} = \vec{\xi}^{\,v} + \mu \cdot \boldsymbol{\mathcal{V}}_{1}^{\top}\,v^{\Delta}\not= \mathbf{0}.
\end{equation*}
Since, we have $\boldsymbol{\mathcal{V}}_{1}^{\top}\,(v-\overline{\mathbf{y}}) \not= \mathbf{0}$ as well as $\boldsymbol{\mathcal{V}}_{1}^{\top}\,v^{\Delta} \not= \mathbf{0}$, and $\boldsymbol{\mathcal{V}}_{1}^{\top}\,v^{\Delta}$ can not be expressed by $-\boldsymbol{\mathcal{V}}_{1}^{\top}\,(v-\overline{\mathbf{y}})$ in accordance with our hypothesis. To see this, suppose that the vector $\Delta$ is expressible in this way, then it must hold $\Delta = -\frac{1}{\mu}\,(\boldsymbol{\EuScript{W}}_{1})^{\dagger}\,\vec{\xi}^{\,v}$. However, this implies
\begin{equation*}
  \boldsymbol{\EuScript{W}}\,\Delta = - \frac{1}{\mu}\,\boldsymbol{\EuScript{W}}\,(\boldsymbol{\EuScript{W}}_{1})^{\dagger}\,\vec{\xi}^{\,v} = - \frac{1}{\mu}\,(\boldsymbol{\mathcal{V}}^{\top}\,\boldsymbol{\EuScript{U}})\,(\boldsymbol{\mathcal{V}}_{1}^{\top}\, \boldsymbol{\EuScript{U}})^{\dagger}\,\vec{\xi}^{\,v} = - \frac{1}{\mu}\,\boldsymbol{\mathcal{V}}^{\top}\,(\boldsymbol{\mathcal{V}}_{1}^{\top})^{\dagger}\,\vec{\xi}^{\,v} \not=\mathbf{0}.
\end{equation*}
\end{enumerate}
This argument terminates the proof.
\end{proof}
 
To complete our uniqueness investigation, we need to establish that the single pre-kernel element of the default game also preserves the pre-nucleolus property for the related games, otherwise we can be sure that there must exist at least a second pre-kernel point for the related game different form the first one. For doing so, we introduce the following set:

\begin{definition}
For every $\mathbf{x} \in \mathbb{R}^{n}$, and $\psi \in \mathbb{R}$ define the set
  \begin{equation}
    \label{eq:dset}
  \mathcal{D}^{v}(\psi,\mathbf{x}) := \left\{S \subseteq N \,\arrowvert\, e^{v}(S,\mathbf{x}) \ge \psi \right\},
  \end{equation}
\end{definition}
\noindent and let $\mathcal{B}=\{S_{1},\ldots, S_{m}\}$ be a collection of non-empty sets of $N$. We denote the collection $\mathcal{B}$ as balanced whenever there exist positive numbers $w_{S}$ for all $S \in \mathcal{B}$ such that we have $\sum_{S \in \mathcal{B}}\, w_{S}\mathbf{1}_{S} = 1_{N}$. The numbers $w_{S}$ are called weights for the balanced collection $\mathcal{B}$ and $\mathbf{1}_{S}$ is the {\bfseries indicator function} or {\bfseries characteristic vector} $\mathbf{1}_{S}:N \mapsto \{0,1\}$ given by $\mathbf{1}_{S}(k):=1$ if $k \in S$, otherwise $\mathbf{1}_{S}(k):=0$.

A characterization of the pre-nucleolus in terms of balanced collections is due to~\citet{kohlb:71}.

\begin{theorem}
  \label{thm:kohlb}
 Let $\langle\, N, v\, \rangle$ be a TU game and let be $\mathbf{x} \in \mathcal{I}^{\,0}(v)$. Then $\mathbf{x} = \nu(N,v) $ if, and only if, for every $\psi \in \mathbb{R}, \mathcal{D}^{v}(\psi,\mathbf{x})\not= \emptyset$ implies that $\mathcal{D}^{v}(\psi,\mathbf{x})$ is a balanced collection over N. 
\end{theorem}

\begin{theorem}
  \label{thm:prnmu}
 Let $\langle\, N, v\, \rangle$ be a TU game that has a singleton pre-kernel such that $\{\mathbf{x}\} = \mathcal{P\text{\itshape r}K}(v) \subset [\vec{\gamma}]$, and let $\langle\, N, v^{\mu}\, \rangle$ be a related game of $v$ derived from $\mathbf{x}$, then $\mathbf{x} = \nu(N,v^{\,\mu})$, whereas the payoff equivalence class $[\vec{\gamma}]$ has non-empty interior. 
\end{theorem}
\begin{proof}
  By our hypothesis, $\mathbf{x} = \nu(N,v)$ is an interior point of an inscribed ellipsoid with maximum volume $\varepsilon := \{\mathbf{y}^{\prime}\, \arrowvert h^{v}_{\gamma}(\mathbf{y}^{\prime}) \le \bar{c} \} \subset [\vec{\gamma}]$, whereas $h_{\gamma}^{v}$ is of type~\eqref{eq:objf2} and $\bar{c} > 0$ (cf.~Lemma~\ref{lem:repl_prk}). This implies by Theorem~\ref{thm:repl_prk} that this point is also a pre-kernel point of game $v^{\mu}$, there is no change in set of lexicographically smallest most effective coalitions $\mathcal{S}(\mathbf{x})$ under $v^{\mu}$. The min-max excess value $\psi^{*}$ obtained by iteratively solving the LP (6.4-6.7) of~\citet[p.~332]{MPSh:79} for game $v$ is smaller than the maximum surpluses derived from $\mathcal{S}(\mathbf{x})$, this implies that there exists a $\bar{\psi} \ge  \psi^{*}$ s.t. $\mathcal{S}(\mathbf{x}) \subseteq \mathcal{D}^{v}(\bar{\psi},\mathbf{x})$, that is, it satisfies Property I of~\citet{kohlb:71}. Moreover, matrix $\mathbf{E}^{\top}$ induced from $\mathcal{S}(\mathbf{x})$ has full rank, therefore, the column vectors of matrix $\mathbf{E}^{\top}$ are a spanning system of $\mathbb{R}^{n}$. Hence, we get $span\,\{\mathbf{1}_{S}\,\arrowvert\, S \in \mathcal{S}(\mathbf{x}) \} = \mathbb{R}^{n}$, which implies that the corresponding matrix $[\mathbf{1}_{S}]_{S \in \mathcal{S}(\mathbf{x})}$ must have rank $n$, therefore collection $\mathcal{S}(\mathbf{x})$ is balanced. In addition, we can choose the largest $\psi \in \mathbb{R}$ s.t. $\emptyset \not= \mathcal{D}^{v}(\psi,\mathbf{x}) \subseteq \mathcal{S}(\mathbf{x})$ is valid, which is a balanced set. Furthermore, we have $\mu\cdot v^{\Delta} \in [-\mathsf{C},\mathsf{C}]^{p^{\prime}}$. Since $\mathsf{C} > 0$,  the set $\mathcal{D}^{v}(\psi - 2\,\mathsf{C},\mathbf{x}) \not=\emptyset$ is balanced as well. Now observe that $e^{v}(S,\mathbf{x}) -\mathsf{C} \le e^{v}(S,\mathbf{x}) + \mu\cdot v^{\Delta}(S) \le e^{v}(S,\mathbf{x}) + \mathsf{C}$ for all $S \subseteq N$. This implies $\mathcal{D}^{v}(\psi,\mathbf{x}) \subseteq \mathcal{S}(\mathbf{x}) \subseteq \mathcal{D}^{v^{\mu}}(\psi - \mathsf{C},\mathbf{x}) \subseteq \mathcal{D}^{v}(\psi - 2\,\mathsf{C},\mathbf{x})$, hence, $ \mathcal{D}^{v^{\mu}}(\psi - \mathsf{C},\mathbf{x})$ is balanced. Let $c \in [-\mathsf{C},\mathsf{C}]$, and from the observation $\lim_{c \uparrow 0}\, \mathcal{D}^{v^{\mu}}(\psi + c,\mathbf{x}) = \mathcal{D}^{v^{\mu}}(\psi,\mathbf{x}) \supseteq \mathcal{D}^{v}(\psi,\mathbf{x})$, we draw the conclusion $\mathbf{x} = \nu(N,v^{\,\mu})$. 
\end{proof}

\begin{theorem}
  \label{thm:siva1}
  Assume that the payoff equivalence class $[\vec{\gamma}]$ induced from TU game $\langle\, N, v\, \rangle$ has non-empty interior. In addition, assume that game $\langle\, N, v\, \rangle$ has a singleton pre-kernel such that $\{\mathbf{x}\} = \mathcal{P\text{\itshape r}K}(v) \subset [\vec{\gamma}]$ is satisfied, then the pre-kernel $\mathcal{P\text{\itshape r}K}(v^{\mu})$ of a related TU game $\langle\, N, v^{\mu}\, \rangle$, as defined by Lemma~\ref{lem:repl_min}, consists of a single point, which is given by $\{\mathbf{x}\} = \mathcal{P\text{\itshape r}K}(v^{\mu})$.
\end{theorem}
\begin{proof}
 This result follows from Theorems~\ref{thm:repl_prk},~\ref{thm:prnmu}, and Propositions~\ref{prop:uniqpk01b},~\ref{prop:uniqpk02}.  
\end{proof}

\begin{example}
  \label{exp:uniqPk}
 In order to illuminate the foregoing discussion of replicating a pre-kernel element consider a four person average-convex but non-convex game that is specified by
\begin{equation*}
  \begin{split}
    v(N) & = 16, v(\{1,2,3\}) = v(\{1,2,4\}) = v(\{1,3,4\}) = 8, \allowdisplaybreaks \\
    v(\{1,3\}) & = 4, v(\{1,4\}) = 1, v(\{1,2\}) = 16/3, v(S)  = 0 \;\text{otherwise},
  \end{split}
\end{equation*}

\noindent with $N=\{1,2,3,4\}$. The pre-kernel coalesces with the pre-nucleolus, which is given by the point: $\nu(v) = \mathcal{P\text{\itshape r}K}(v)= (44/9,4,32/9,32/9)$. Obviously, the set $\mathcal{S}(\nu(v))=\{\{2\},\{3\},\{4\},\{1,2\},\{1,3,4\}\}$ is balanced, form this set a boundary vector $\vec{b}=(4,32/9,32/9,80/9,12)$ is obtained by $\nu(v)(S)$ for $S \in \mathcal{S}(\nu(v))$. Define matrix $\mathbf{A}$ by $[\mathbf{1}_{S}]_{S \in \mathcal{S}(\nu(v))}$, then the solution of the system $\mathbf{A}\,\mathbf{x} = \vec{b}$ reproduces the pre-nucleolus. Moreover, this imputation is even an interior point, thus the non-empty interior condition is valid. Hence, by Theorem~\ref{thm:repl_prk} a redistribution of the bargaining power among coalitions can be attained while supporting the imputation $(44/9,4,32/9,32/9)$ still as a pre-kernel element for a set of related games. In order to get a null space $\mathcal{N}_{\boldsymbol{\EuScript{W}}}$ with maximum dimension we set the parameter $\mu$ to $0.9$. In this case, the rank of matrix $\boldsymbol{\EuScript{W}}$ must be equal to $4$, and we could derive at most $11$-linear independent games which replicate the point $(44/9,4,32/9,32/9)$ as a pre-kernel element. Theorem~\ref{thm:siva1} even states that this point is also the sole pre-kernel point, hence the pre-kernel coincide with the pre-nucleolus for these games (see Table~\ref{tab:rpl_acvt1}).

Notice that non of these $11$-linear independent related games is average-convex. Only two games, namely $v_{1}$ and $v_{3}$ are zero-monotonic and super-additive. Nevertheless, all games have a non-empty core and are semi-convex. The cores of the games have between $16$ and $24$-vertices, and have volumes that range from approximately $80$ to $127$ percent of the default core. TU game $v_{2}$ has the smallest and $v_{3}$ the largest core.\footnote{The example can be reproduced while using our MATLAB toolbox {\itshape MatTuGames}~\citeyear{mei:11}. The results can also be verified with our Mathematica package {\itshape TuGames}~\citeyear{mei:10a}.}\hfill$\#$
\end{example}
{\scriptsize
\begin{center}
\begin{ThreePartTable}
  \begin{TableNotes}
   \item[a] \label{tn:a} Pre-Kernel and Pre-Nucleolus: $(44/9,4,32/9,32/9)$
   \item[b] \label{tn:b} ACV: Average-Convex Game
   \item[c] \label{tn:c} ZM: Zero-Monotonic Game
   \item[d] Note: Computation performed with MatTuGames.
  \end{TableNotes}
\begin{longtable}[c]{*{3}{ccccccccc}}
\caption[List of Games\tnote{d}~~which possess the same unique Pre-Kernel\tnote{a}~~as~$v$]{List of Games\tnote{d}~~which possess the same unique Pre-Kernel\tnote{a}~~as~$v$} \label{grid_mlmmh} \\[.3em]

\hline \multicolumn{1}{c}{} &\multicolumn{1}{c}{} & \multicolumn{1}{c}{} & \multicolumn{1}{c}{} & \multicolumn{1}{c}{$\mu=0.9$} & \multicolumn{1}{c}{} & \multicolumn{1}{c}{} & \multicolumn{1}{c}{} & \multicolumn{1}{c}{} \\ 
\multicolumn{1}{c}{Game} &\multicolumn{1}{c}{$\{1\}$} & \multicolumn{1}{c}{$\{2\}$} & \multicolumn{1}{c}{$\{1,2\}$} & \multicolumn{1}{c}{$\{3\}$} & \multicolumn{1}{c}{$\{1,3\}$} & \multicolumn{1}{c}{$\{2,3\}$} & \multicolumn{1}{c}{$\{1,2,3\}$} & \multicolumn{1}{c}{$\{4\}$} \\[.3em] \hline\hline 
\endfirsthead

\multicolumn{10}{c}%
{{\bfseries \tablename\ \thetable{} -- continued from previous page}} \\
\hline \multicolumn{1}{c}{} &
\multicolumn{1}{c}{} & 
\multicolumn{1}{c}{} & 
\multicolumn{1}{c}{} & 
\multicolumn{1}{c}{$\mu=0.9$} & 
\multicolumn{1}{c}{} & 
\multicolumn{1}{c}{} & 
\multicolumn{1}{c}{} & 
\multicolumn{1}{c}{} \\
\multicolumn{1}{c}{Game} &
\multicolumn{1}{c}{$\{1,4\}$} &
\multicolumn{1}{c}{$\{2,4\}$} &
\multicolumn{1}{c}{$\{1,2,4\}$} &
\multicolumn{1}{c}{$\{3,4\}$} &
\multicolumn{1}{c}{$\{1,3,4\}$} &
\multicolumn{1}{c}{$\{2,3,4\}$} &
\multicolumn{1}{c}{$N$} &
\multicolumn{1}{c}{ACV~\tnote{b}} &
\multicolumn{1}{c}{ZM~\tnote{c}} & \\[.3em] \hline\hline 
\endhead

\hline \multicolumn{9}{r}{{Continued on next page}} \\ \hline\hline
\endfoot
\insertTableNotes\\
\hline \hline
\endlastfoot
$v$    &   0      &    0      &   16/3    &     0     &    4      &     0     &     8     &       0  \\
$v_{1}$ &  18/49    &  32/95    &  127/24   &  -1/24    &  256/59   &   4/13    &   175/22  &   -1/24 \\
$v_{2}$ &  -9/25    &  21/38    &   89/16   &   11/48   &  231/58   &  42/71    &   385/47  &   11/48 \\
$v_{3}$ &  -14/45   &  -1/40    &  201/41   &  -28/65   &   39/11   &  -19/44   &   142/19  &   -28/65 \\
$v_{4}$ &    0      &    0      &   16/3    &     0     &  159/47   &   16/33   &   107/14  &      0   \\
$v_{5}$ &    0      &    0      &   16/3    &     0     &  149/40   &  -37/102  &   497/66  &      0   \\
$v_{6}$ &    0      &    0      &   16/3    &     0     &   4       &   -5/47   &   143/19  &      0   \\
$v_{7}$ &    0      &    0      &   16/3    &     0     &   4       &   -5/47   &   143/19  &      0   \\
$v_{8}$ &    0      &    0      &   16/3    &     0     &  149/40   &  -37/102  &   497/66  &      0   \\
$v_{9}$ &    0      &    0      &   16/3    &     0     &  149/40   &  -37/102  &   497/66  &      0   \\
$v_{10}$ &    0      &    0      &   16/3    &     0     &    4      &   -5/47   &   143/19  &      0   \\
$v_{11}$ &   0      &    0      &   16/3    &     0     &    4      &   -5/47   &   143/19  &       0  \\
\pagebreak
$v$     &    1     &    0      &    8    &    0       &      8    &     0  &  16  & Y & Y \\
$v_{1}$  &  79/59   &   4/13    &  175/22 &  -4/57     &   792/95  &  10/33 &  16  & N & Y \\    
$v_{2}$  &  57/58   &  42/71    &  385/47 &   4/7      &   325/38  &  31/56 &  16  & N & N \\    
$v_{3}$  &   6/11   & -19/44    &  142/19 &  -27/47    &   319/40  & -29/55 &  16  & N & Y \\    
$v_{4}$  &  41/34   &  -3/46    &  428/53 &   7/34     &      8    &  14/25 &  16  & N & N \\    
$v_{5}$  & 203/120  &   2/41    &  167/19 &  -5/24     &      8    &  -9/19 &  16  & N & N \\    
$v_{6}$  &    1     &  23/29    &  139/16 &    0       &      8    &  18/31 &  16  & N & N \\    
$v_{7}$  &    1     & -5/47     &  139/16 &    0       &      8    &  -8/25 &  16  & N & N \\    
$v_{8}$  &  19/24   &  2/41     &   71/9  &  83/120    &      8    &  26/61 &  16  & N & N \\    
$v_{9}$  &  19/24   &  2/41     &   71/9  &  -5/24     &      8    &  -9/19 &  16  & N & N \\    
$v_{10}$ &    1     & -5/47     &  475/61 &    0       &      8    &  18/31 &  16  & N & N \\    
$v_{11}$ &    1     & -5/47     &  475/61 &    0       &      8    &  -8/25 &  16  & N & N \\ \hline\hline
\label{tab:rpl_acvt1}
\end{longtable}
\end{ThreePartTable}
\end{center}
}

\section{On the Continuity of the Pre-Kernel}
\label{sec:lhc}
In the previous section, we have established uniqueness on the set of related games. Here, we generalize these results while showing that even on the convex hull comprising the default and related games in the game space, the pre-kernel must be unique and is identical with the point specified by the default game. Furthermore, the pre-kernel correspondence restricted on this convex subset in the game space must be single-valued, and therefore continuous.   

Recall that the relevant game space is defined through $\mathcal{G}(N) := \{v \in \mathcal{G}^{n}\,\arrowvert\, v(\emptyset) = 0 \land v(N) > 0\}$, and 

\begin{equation*}
   \mathcal{G}_{\mu,v}^{n}:=\left\{v^{\mu} \in \mathcal{G}(N)\, \arrowvert\, \mu\cdot v^{\Delta} \in [-\mathsf{C},\mathsf{C}]^{p^{\prime}} \right\}.
\end{equation*}
This set is the translate of a convex set by $v$, which is also convex and non-empty with dimension $p^{\prime}-m^{\prime}$, if matrix $\boldsymbol{\EuScript{W}}$ has rank $m^{\prime} \le q < p^{\prime}$. Then we can construct a convex set in the game space $\mathcal{G}(N)$ by taking the convex hull of game $v$ and the convex set $\mathcal{G}_{\mu,v}^{n}$, thus  
 
\begin{equation*}
  \mathcal{G}_{c}^{n} := conv\; \{v,\mathcal{G}_{\mu,v}^{n}\}. 
\end{equation*}
 
\begin{theorem}
 \label{thm:siva2}
 The pre-kernel $\mathcal{P\text{\itshape r}K}(v^{\mu^{*}})$ of game $v^{\mu^{*}}$ belonging to $\mathcal{G}_{c}^{n}$ is a singleton, and is equal to $\{\mathbf{x}\} = \mathcal{P\text{\itshape r}K}(v)$.
\end{theorem}
\begin{proof}
  Let be $\{\mathbf{x}\} = \mathcal{P\text{\itshape r}K}(v)$ for game $v$. Take a convex combination of games in $\mathcal{G}_{c}^{n}$, hence
  \begin{equation*}
    v^{\mu^{*}} = \sum_{k=1}^{m} t_{k}\cdot v_{k}^{\mu} + t_{m+1}\cdot v = \sum_{k=1}^{m} t_{k}\cdot (v + \mu\cdot v_{k}^{\Delta}) + t_{m+1} \cdot v = v + \mu \sum_{k=1}^{m} t_{k}\cdot v_{k}^{\Delta} + \mu\;t_{m+1} \cdot \mathbf{0} = v + \mu \cdot v^{\Delta^{*}}, 
  \end{equation*}
with $v^{\Delta^{*}} := \sum_{k=1}^{m} t_{k}\cdot v_{k}^{\Delta} + t_{m+1} \cdot \mathbf{0}$, where $0 \le t_{k} \le 1, \forall k \in  \{1,2, \ldots ,m+1\}$, and $\sum_{k=1}^{m+1} t_{k} =1$. Then $\mu\,v^{\Delta^{*}} \in [-\mathsf{C},\mathsf{C}]^{p^{\prime}}$, thus the set of lexicographically smallest coalitions $\mathcal{S}(\mathbf{x})$ does not change. By Theorem~\ref{thm:repl_prk} the vector $\{\mathbf{x}\} = \mathcal{P\text{\itshape r}K}(v)$ is also a pre-kernel element of game $v^{\mu^{*}}$. But then by Theorem~\ref{thm:siva1} the pre-kernel of game $v^{\mu^{*}}$ consists of a single point, therefore $\{\mathbf{x}\} = \mathcal{P\text{\itshape r}K}(v^{\mu^{*}})$. 
\end{proof}

\begin{example}
  \label{exp:siva1}
To see that even on the convex hull $\mathcal{G}_{c}^{4}$, which is constituted by the default and related games of Table~\ref{tab:rpl_acvt1}, a particular TU game has the same singleton pre-kernel, we choose the following vector of scalars $\vec{t}=(1,3,8,1,2,4,3,5,7,9,2,3)/48$ such that $\sum_{k=1}^{12} t_k=1$ is given to construct by the convex combination of games presented by Table~\ref{tab:rpl_acvt1} a TU game $v^{\mu^{*}}$ that reproduces the imputation $(44/9,4,32/9,32/9)$ as its unique pre-kernel. The TU game $v^{\mu^{*}}$ on this convex hull in the game space that replicates this pre-kernel is listed through Table~\ref{tab:siva1}:   
\begin{center}
\begin{threeparttable}
{\footnotesize
\setlength{\tabcolsep}{.3cm}
\caption{A TU Game $v^{\mu^{*}}$ on $\mathcal{G}_{c}^{4}$ with the same singleton Pre-Kernel as $v$~\tnote{a,b}}
\begin{tabular}[c]{c c c c c c c c}
\hline
$S$  & $v^{\mu^{*}}\!(S)$ & $S$      & $v^{\mu^{*}}\!(S)$ & $S$       & $v^{\mu^{*}}\!(S)$ & $S$    & $v^{\mu^{*}}\!(S)$  \\[.1em] \hline\hline 
$\{1\}$  & $-1/23$    & $\{1,2\}$& $134/25$     & $\{2,4\}$   & $173/1125$ & $\{1,3,4\}$& $576/71$  \\[.1em]
$\{2\}$  & $8/71$     & $\{1,3\}$& $530/137$    & $\{3,4\}$   & $19/144$   & $\{2,3,4\}$& $15/232$  \\[.1em]
$\{3\}$  & $2/75$     & $\{1,4\}$& $179/178$    & $\{1,2,3\}$ & $1436/187$ & $N$        & $16$ \\[.1em]
$\{4\}$  & $2/75$     & $\{2,3\}$& $-8/157$     & $\{1,2,4\}$ & $1946/239$ &            & \\[.1em] \hline\hline
\end{tabular}
\label{tab:siva1}
\begin{tablenotes}
\item[a] Pre-Kernel and Pre-Nucleolus: $(44/9,4,32/9,32/9)$
\item[b] Note: Computation performed with MatTuGames.
\end{tablenotes}
}
\end{threeparttable}
\end{center}
This game is neither average-convex nor zero-monotonic, however, it is again semi-convex and has a rather large core with a core volume of $97$ percent w.r.t. the core of the average-convex game, and $20$ vertices in contrast to $16$ vertices respectively. \hfill$\#$
\end{example}

Let $\mathcal{X}$ and $\mathcal{Y}$ be two metric spaces. A set-valued function or correspondence $\sigma$ of $\mathcal{X}$ into $\mathcal{Y}$ is a rule that assigns to every element $x \in \mathcal{X}$ a non-empty subset $\sigma(x) \subset \mathcal{Y}$. Given a correspondence $\sigma: \mathcal{X} \twoheadrightarrow \mathcal{Y}$, the corresponding graph of $\sigma$ is defined by 

\begin{equation}
  \label{eq:gr}
  Gr(\sigma) := \left\{(x,y) \in \mathcal{X} \times \mathcal{Y} \,\arrowvert\, y \in \sigma(x)  \right\}.
\end{equation}

\begin{definition}
 \label{def:grcl}
  A set-valued function $\sigma : \mathcal{X} \twoheadrightarrow \mathcal{Y} $ is closed, if $Gr(\sigma)$ is a closed subset of $\mathcal{X} \times \mathcal{Y}$  
\end{definition}

The graph of the pre-kernel correspondence $\mathcal{P\text{\itshape r}K}$ is given by 
\begin{equation*}
 \begin{split}
  Gr(\mathcal{P\text{\itshape r}K}) := \left\{ (v,\mathbf{x})\,\arrowvert\, v \in \mathcal{G}^{n}, \mathbf{x} \in \mathcal{I}^{0}(v), \;\; s_{ij}(\mathbf{x},v) = s_{ji}(\mathbf{x},v) \quad\text{for all}\; i,j \in N, i\neq j \,\right\}.
 \end{split}
\end{equation*}
Similar, the graph of the solution set of function $h^{v}$ of type~\eqref{eq:objfh} is specified by 
\begin{equation*}
 \begin{split}
  Gr(M(h^{v})) & := \left\{ (v,\mathbf{x})\,\arrowvert\, v \in \mathcal{G}^{n}, \mathbf{x} \in \mathcal{I}^{0}(v), \;\; h^{v}(\mathbf{x}) = 0 \,\right\} \allowdisplaybreaks \\
              & = \bigcup_{k \in \mathcal{J}^{\prime}}\; \left\{ (v,\mathbf{x})\,\arrowvert\, v \in \mathcal{G}^{n}, \mathbf{x} \in \overline{[\vec{\gamma}_{k}]}, \;\; h_{\gamma_{k}}^{v}(\mathbf{x}) = 0 \,\right\} = \bigcup_{k \in \mathcal{J}^{\prime}}\; Gr(M(h^{v}_{\gamma_{k}}, \overline{[\vec{\gamma}_{k}]})),
 \end{split}
\end{equation*}
with $ \mathcal{J}^{\prime} : = \{k \in \mathcal{J}\, \arrowvert\, g(\vec{\gamma}_{k}) = 0\}$. This graph is equal to the finite union of graphs of the restricted solution sets of quadratic and convex functions $h^{v}_{\gamma_{k}}$ of type~\eqref{eq:objf2}. The restriction of each solution set of function $h^{v}_{\gamma_{k}}$ to $\overline{[\vec{\gamma}_{k}]}$ is bounded, closed, and convex (cf.~\citet[Lemmata 7.1.3, 7.3.1]{mei:13}), hence each graph $Gr(M(h^{v}_{\gamma_{k}}, \overline{[\vec{\gamma}_{k}]}))$ from the finite index set $\mathcal{J}^{\prime}$ is bounded, closed and convex. 

\begin{proposition}
  \label{prop:eqgr}
  The following relations are satisfied between the above graphs:
 \begin{equation}
   \label{eq:eqgr}
   Gr(\mathcal{P\text{\itshape r}K}) = Gr(M(h^{v})) =  \bigcup_{k \in \mathcal{J}^{\prime}}\; Gr(M(h^{v}_{\gamma_{k}}, \overline{[\vec{\gamma}_{k}]})). 
 \end{equation}
Hence, the pre-kernel correspondence $\mathcal{P\text{\itshape r}K}:\mathcal{G}(N) \twoheadrightarrow \mathbb{R}^{n}$ is closed and bounded.
\end{proposition}
\begin{proof}
  The equality of the graph of the pre-kernel and the solution set of function $h^{v}$ follows in accordance with Corollary~\ref{cor:rep}. Finally, the last equality is a consequence of Theorem 7.3.1 by~\citet{mei:13}. From this argument boundedness and closedness follows. 
\end{proof}

\begin{definition}
  \label{def:uhc}
 The correspondence $\sigma: \mathcal{X} \twoheadrightarrow \mathcal{Y}$ is said to be upper hemi-continuous ({\bfseries uhc}) at $x$ if for every open set $\mathcal{O}$ containing $\sigma(x) \subseteq \mathcal{O}$ it exists an open set $\mathcal{Q} \subseteq \mathcal{Y}$ of $x$ such that $\sigma(x^{\prime}) \subseteq \mathcal{O}$ for every $x^{\prime} \in \mathcal{Q}$. The correspondence $\sigma$ is {\bfseries uhc}, if it is {\bfseries uhc} for each $x \in \mathcal{X}$.
\end{definition}

\begin{definition}
  \label{def:lhc}
The correspondence $\sigma: \mathcal{X} \twoheadrightarrow \mathcal{Y}$ is said to be lower hemi-continuous ({\bfseries lhc}) at $x$ if for every open set $\mathcal{O}$ in $\mathcal{Y}$ with $\sigma(x) \cap \mathcal{O} \not= \emptyset$ it exists an open set $\mathcal{Q} \subseteq \mathcal{Y}$ of $x$ such that $\sigma(x^{\prime}) \cap \mathcal{O} \not= \emptyset$ for every $x^{\prime} \in \mathcal{Q}$. The correspondence $\sigma$ is {\bfseries lhc}, if it is {\bfseries lhc} for each $x \in \mathcal{X}$.
\end{definition}

\begin{lemma} 
  \label{lem:lhc}
 Let $\mathcal{X}$ be a non-empty and convex polyhedral subset of $\mathbb{R}^{\tilde{p}}$, and $\mathcal{Y} \subseteq \mathbb{R}^{\tilde{n}}$. If $\sigma: \mathcal{X} \twoheadrightarrow \mathcal{Y}$ is a bounded correspondence with a convex graph, then $\sigma$ is lower hemi-continuous.
\end{lemma}
\begin{proof}
  For a proof see \citet[pp. 185-186]{pel_sud:07}.
\end{proof}

\begin{theorem}
  \label{thm:cont}
  The pre-kernel correspondence $\mathcal{P\text{\itshape r}K}:\mathcal{G}(N)  \twoheadrightarrow \mathbb{R}^{n}$ is on $\mathcal{G}_{c}^{n}$ upper hemi-continuous as well as lower hemi-continuous, that is, continuous. 
\end{theorem}
\begin{proof}
  The non-empty set $\mathcal{G}_{c}^{n}$ is a bounded polyhedral set, which is convex by construction. We draw from Proposition~\ref{prop:eqgr} the conclusion that the graph of the pre-kernel correspondence is bounded and closed. Form Theorem~\ref{thm:siva2} it follows $\arrowvert\, \mathcal{J}^{\prime}\, \arrowvert =1$ on $\mathcal{G}_{c}^{n}$, this implies that the graph of the pre-kernel correspondence is also convex on $\mathcal{G}_{c}^{n}$. The sufficient conditions of Lemma~\ref{lem:lhc} are satisfied, hence $\mathcal{P\text{\itshape r}K}$ is lower hemi-continuous on $\mathcal{G}_{c}^{n}$. It is known from Theorem 9.1.7.~by \citet{pel_sud:07} that $\mathcal{P\text{\itshape r}K}$ is upper hemi-continuous on $\mathcal{G}(N)$. Hence, on the restricted set $\mathcal{G}_{c}^{n}$, the set-valued function $\mathcal{P\text{\itshape r}K}$ is upper and lower hemi-continuous, and therefore continuous. Actual, it is a continuous function on $\mathcal{G}_{c}^{n}$ in view of $\arrowvert\, \mathcal{J}^{\prime}\, \arrowvert =1$.
\end{proof} 

\begin{corollary}
  \label{cor:prkf}
  The pre-kernel correspondence $\mathcal{P\text{\itshape r}K}:\mathcal{G}(N) \twoheadrightarrow \mathbb{R}^{n}$ is on $\mathcal{G}_{c}^{n}$ single-valued and constant. 
\end{corollary}

\begin{example}
  \label{exp:nlhc}
To observe that on the restricted set $\mathcal{G}_{c}^{4}$ the pre-kernel correspondence $\mathcal{P\text{\itshape r}K}:\mathcal{G}(N) \twoheadrightarrow \mathbb{R}^{n}$ is single-valued and continuous, we exemplarily select a line segment in $\mathcal{G}_{c}^{4}$ to establish that all games on this segment have the same singleton pre-kernel. For this purpose, we resume Example~\ref{exp:uniqPk} and~\ref{exp:siva1}. Then we choose a vector of scalars $\vec{t}^{\epsilon}:=(1,3,8,1,2,4+\epsilon,3,5,7,9,2-\epsilon,3)/48$ with $t^{\epsilon}_{k} \ge 0$ for each $k$ such that $\sum_{k=0}^{11} t^{\epsilon}_k=1$ and $\epsilon \in [-2,2]$. Thus, we define the line segment in $\mathcal{G}_{c}^{4}$ through TU game $v^{\mu^{*}}$ from Example~\ref{exp:siva1} by 
\begin{equation*}
  \mathcal{G}_{c}^{4,l} := \bigg\{\sum_{k=0}^{11} t^{\epsilon}_{k}\cdot v_{k}^{\mu}\;\bigg\arrowvert v_{k}^{\mu} \in \mathcal{G}_{c}^{4}, \epsilon \in [-2,2] \bigg\}.
\end{equation*}
Therefore, for each game in the line segment $\mathcal{G}_{c}^{4,l}$, we can write
\begin{equation*}
  \begin{split}
  v^{\epsilon} & := \sum_{k=1}^{11} t^{\epsilon}_{k}\cdot v_{k}^{\mu} + t^{\epsilon}_{0}\cdot v = \sum_{k=1}^{11} t_{k}\cdot v_{k}^{\mu} + t_{0}\cdot v + \frac{\epsilon}{48}\,(v_{6}^{\mu} - v_{11}^{\mu}) = v^{\mu^{*}} + \frac{\epsilon}{48}\,(v_{6}^{\mu} - v_{11}^{\mu})\allowdisplaybreaks\\
             & = v + \mu \cdot v^{\Delta^{*}} + \frac{\epsilon\,\mu}{48}\,(v_{6}^{\Delta} - v_{11}^{\Delta}).
 \end{split}
\end{equation*}
We extend the pre-kernel element $\mathbf{x} = (44/9,4,32/9,32/9)$ to a vector $\overline{\mathbf{x}}$ in order to define the excess vector under game $v$ as $\overline{e}:=v-\overline{\mathbf{x}}$, and for game $v^{\epsilon}$ as $\overline{e}^{\,v^{\epsilon}}:=v^{\epsilon}-\overline{\mathbf{x}}$, respectively. According to these definitions, we get for $\vec{\zeta}^{v^{\epsilon}} = \vec{\xi}^{v^{\epsilon}}$ at $\mathbf{x}$ the following chain of equalities:
\begin{equation*}
  \vec{\xi}^{v^{\epsilon}} = \boldsymbol{\mathcal{V}}^{\top}\,\overline{e}^{\,v^{\epsilon}} = \boldsymbol{\mathcal{V}}^{\top}\,\big(v-\overline{\mathbf{x}} + \mu \cdot v^{\Delta^{*}} + \frac{\epsilon\,\mu}{48}\,(v_{6}^{\Delta} - v_{11}^{\Delta})\big) = \boldsymbol{\mathcal{V}}^{\top}\,(v-\overline{\mathbf{x}}) = \boldsymbol{\mathcal{V}}^{\top}\,\overline{e} = \vec{\xi} = \vec{\zeta} = \mathbf{0}, 
\end{equation*}
The last equality is satisfied, since $\mathbf{x}$ is the pre-kernel of game $v$. Recall that it holds $\mu\,v^{\Delta^{*}},\mu\,v_{6}^{\Delta},\mu\,v_{11}^{\Delta} \in [-\mathsf{C},\mathsf{C}]^{15}$, whereas $\boldsymbol{\mathcal{V}}^{\top} \,v^{\Delta^{*}} = \boldsymbol{\mathcal{V}}^{\top}\,v_{6}^{\Delta} = \boldsymbol{\mathcal{V}}^{\top}\,v_{11}^{\Delta} = \mathbf{0}$ is in force. Therefore, for each TU game $v^{\epsilon} \in \mathcal{G}_{c}^{4,l}$ we attain
\begin{equation*}
  \mathcal{P\text{\itshape r}K}(v^{\epsilon}) = (44/9,4,32/9,32/9).
\end{equation*}
The pre-kernel correspondence $\mathcal{P\text{\itshape r}K}$ is a single-valued and constant mapping on $\mathcal{G}_{c}^{4,l}$. Hence its is continuous on the restriction $\mathcal{G}_{c}^{4,l}$, and due to Theorem~\ref{thm:cont} a fortiori on $\mathcal{G}_{c}^{4}$. \hfill$\#$
\end{example}

\section{Preserving the Pre-Nucleolus Property}
\label{sec:prspn}
In this section we study some conditions under which a pre-nucleolus of a default game can preserve the pre-nucleolus property in order to generalize the above results in the sense to identify related games with an unique pre-kernel point even when the default game has not a single pre-kernel point. This question can only be addressed with limitation, since we are not able to make it explicit while giving only sufficient conditions under which the pre-kernel point must be at least disconnected, otherwise it must be a singleton. However, a great deal of our investigation is devoted to work out explicit conditions under which the pre-nucleolus of a default game will loose this property under a related game.

For the next result remember that a balanced collection $\mathcal{B}$ is called minimal balanced, if it does not contain a proper balanced sub-collection.  

\begin{theorem}
  \label{thm:nprnmu}
 Let $\langle\, N, v\, \rangle$ be a TU game that has a non unique pre-kernel such that $\mathbf{x} \in \mathcal{P\text{\itshape r}K}(v)$, $\mathbf{y} = \nu(v)$ with $\mathbf{x},\mathbf{y}  \in [\vec{\gamma}]_{v}$, and $\mathbf{x} \not= \mathbf{y}$ is satisfied. In addition, let $\langle\, N, v^{\mu}\, \rangle$ be a related game of $v$ with $\mu \not=0$ derived from $\mathbf{x}$ such that $\mathbf{x} \in \mathcal{P\text{\itshape r}K}(v^{\mu}) \cap [\vec{\gamma}]_{v^{\mu}}$, and $\mathbf{y} \not\in [\vec{\gamma}]_{v^{\mu}}$ holds. If the collection $\mathcal{S}^{v}(\mathbf{x})$ as well as its sub-collections are not balanced, 
 \begin{enumerate}
 \item then $\mathbf{y} \not\in \mathcal{P\text{\itshape r}N}(v^{\mu})$.
 \item Moreover, if in addition $\mathbf{x} = \mathbf{y} \not\in [\vec{\gamma}]_{v^{\mu}}$, then $\mathbf{x} \not\in \mathcal{P\text{\itshape r}N}(v^{\mu})$. 
 \end{enumerate}
\end{theorem}
\begin{proof}
The proof starts with the first assertion.
\begin{enumerate}
  \item By our hypothesis, $\mathbf{x}$ is a pre-kernel element of game $v$ and a related game $v^{\mu}$ that is derived from $\mathbf{x}$. There is no change in set of lexicographically smallest most effective coalitions $\mathcal{S}^{v}(\mathbf{x})$ under $v^{\mu}$ due to $\mathbf{x}  \in [\vec{\gamma}]_{v^{\mu}}$, hence $\mathcal{S}^{v}(\mathbf{x}) = \mathcal{S}^{v^{\mu}}(\mathbf{x})$. Moreover, we have $\mu\cdot v^{\Delta} \in \mathbb{R}^{p^{\prime}}$. Furthermore, it holds $\mathbf{y} = \nu(v)$ by our assumption. Choose a balanced collection $\mathcal{B}$ that contains $\mathcal{S}^{v}(\mathbf{x})$ such that $\mathcal{B}$ is minimal. Then single out any $\psi \in \mathbb{R}$ such that the balanced set $\mathcal{D}^{v}(\psi,\mathbf{y})$ satisfies $\mathcal{S}^{v}(\mathbf{x}) \subseteq \mathcal{B} \subseteq \mathcal{D}^{v}(\psi,\mathbf{y}) \not=\emptyset$. Now choose $\epsilon > 0$ such that $\mathcal{D}^{v}(\psi,\mathbf{y}) = \mathcal{D}^{v}(\psi - 2\,\epsilon,\mathbf{y})$ is given. The set $\mathcal{D}^{v}(\psi - 2\,\epsilon,\mathbf{y})$ is balanced as well. Observe that due to $\mathbf{x}  \in [\vec{\gamma}]_{v^{\mu}}$ we get $\mu\cdot v^{\Delta}(S) \le \epsilon $ for all $S \subset N$. However, it exists some coalitions $S \in \mathcal{S}^{v}(\mathbf{x})$ such that $e^{v}(S,\mathbf{y}) - \epsilon \not\le e^{v}(S,\mathbf{y}) + \mu\cdot v^{\Delta}(S)$ holds. Let $c \in [-\epsilon,\epsilon]$, now as $\lim_{c \uparrow 0}\, \mathcal{D}^{v^{\mu}}(\psi + c,\mathbf{y}) = \mathcal{D}^{v^{\mu}}(\psi,\mathbf{y})$ we have $\mathcal{D}^{v^{\mu}}(\psi,\mathbf{y}) \subseteq \mathcal{D}^{v}(\psi,\mathbf{y})$. Furthermore, we draw the conclusion that $\mathcal{S}^{v}(\mathbf{x}) \not\subseteq \mathcal{D}^{v^{\mu}}(\psi, \mathbf{y})$ is given due to $\mathcal{S}^{v}(\mathbf{x}) = \mathcal{S}^{v}(\mathbf{y}) \not= \mathcal{S}^{v^{\mu}}(\mathbf{y})$. Therefore, we obtain $\mathcal{D}^{v^{\mu}}(\psi,\mathbf{y}) \subset \mathcal {B} \subseteq \mathcal{D}^{v}(\psi - 2\,\epsilon,\mathbf{y})$. To see this, assume that $\mathcal{D}^{v^{\mu}}(\psi, \mathbf{y})$ is balanced, then we get $\mathcal{B} \subseteq \mathcal{D}^{v^{\mu}}(\psi, \mathbf{y})$, since $\mathcal{B}$ is minimal balanced. This implies $\mathcal{S}^{v}(\mathbf{x}) \subseteq \mathcal{D}^{v^{\mu}}(\psi, \mathbf{y})$. However, this contradicts $\mathcal{S}^{v}(\mathbf{x}) \not\subseteq \mathcal{D}^{v^{\mu}}(\psi, \mathbf{y})$. We conclude that $\mathcal{D}^{v^{\mu}}(\psi,\mathbf{y}) \subset \mathcal {B}$ must hold, but then the set $\mathcal{D}^{v^{\mu}}(\psi,\mathbf{y})$ can not be balanced. Hence, $\mathbf{y} \not\in \mathcal{P\text{\itshape r}N}(v^{\mu})$.  

 \item Finally, if $\mathbf{x}=\mathbf{y}$, then $\mathbf{x}$ is the pre-nucleolus of game $v$, but it does not belong anymore to payoff equivalence class $[\vec{\gamma}]$ under $v^{\mu}$, that is, $[\vec{\gamma}]$ has shrunk. Therefore, $\mathcal{S}^{v}(\mathbf{x}) \not= \mathcal{S}^{v^{\mu}}(\mathbf{x}) $. Define from the set $\mathcal{S}^{v}(\mathbf{x})$ a minimal balanced collection $\mathcal{B}$ that contains $\mathcal{S}^{v}(\mathbf{x})$. In the next step, we can single out any $\psi \in \mathbb{R}$ such that the balanced set $\mathcal{D}^{v}(\psi,\mathbf{x})$ satisfies $\mathcal{S}^{v}(\mathbf{x}) \subseteq \mathcal{B} \subseteq \mathcal{D}^{v}(\psi,\mathbf{x}) \not=\emptyset$. In view of $\mathbf{x} \in \mathcal{P\text{\itshape r}K}(v^{\mu})$, it must exist an $\epsilon > 0$ within the maximum surpluses can be varied without effecting the pre-kernel property of $\mathbf{x}$ even when $\mathbf{x}  \not\in [\vec{\gamma}]_{v^{\mu}}$, thus we have $\mu\cdot v^{\Delta}(S) \le \epsilon $ for all $S \subset N$. This implies that $\mathcal{D}^{v}(\psi,\mathbf{x}) \subseteq \mathcal{D}^{v}(\psi - 2\,\epsilon,\mathbf{x})$ is in force. The set $\mathcal{D}^{v}(\psi - 2\,\epsilon,\mathbf{x})$ is balanced as well. However, it exists some coalitions $S \in \mathcal{S}^{v}(\mathbf{x})$ such that $e^{v}(S,\mathbf{x}) - \epsilon \not\le e^{v}(S,\mathbf{x}) + \mu\cdot v^{\Delta}(S)$ is valid. Let $c \in [-\epsilon,\epsilon]$, now as $\lim_{c \uparrow 0}\, \mathcal{D}^{v^{\mu}}(\psi + c,\mathbf{x}) = \mathcal{D}^{v^{\mu}}(\psi,\mathbf{x})$ we have $\mathcal{D}^{v^{\mu}}(\psi,\mathbf{x}) \subseteq \mathcal{D}^{v}(\psi,\mathbf{x})$. Furthermore, we draw the conclusion that $\mathcal{S}^{v}(\mathbf{x}) \not\subseteq \mathcal{D}^{v^{\mu}}(\psi, \mathbf{x})$ is given due to $\mathcal{S}^{v}(\mathbf{x}) \not= \mathcal{S}^{v^{\mu}}(\mathbf{x})$. Therefore, we obtain $\mathcal{D}^{v^{\mu}}(\psi,\mathbf{x}) \subset \mathcal {B} \subseteq \mathcal{D}^{v}(\psi - 2\,\epsilon,\mathbf{x})$ by the same reasoning as under (1). Then the set $\mathcal{D}^{v^{\mu}}(\psi,\mathbf{x})$ can not be balanced. Hence, $\mathbf{x} \not\in \mathcal{P\text{\itshape r}N}(v^{\mu})$.  
 \end{enumerate}
\end{proof}

\begin{theorem}
  \label{thm:nprnmu2}
 Let $\langle\, N, v\, \rangle$ be a TU game that has a non unique pre-kernel such that $\mathbf{x} \in \mathcal{P\text{\itshape r}K}(v) \cap [\vec{\gamma}]$, $\{\mathbf{y}\} = \mathcal{P\text{\itshape r}N}(v) \cap [\vec{\gamma}_{1}]$ is satisfied, and let $\langle\, N, v^{\mu}\, \rangle$ be a related game of $v$ with $\mu \not=0$ derived from $\mathbf{x}$ such that $\mathbf{x} \in \mathcal{P\text{\itshape r}K}(v^{\mu}) \cap [\vec{\gamma}]$ holds. If $\Delta \in \mathcal{N}_{\boldsymbol{\EuScript{W}}} \backslash \mathcal{N}_{\boldsymbol{\EuScript{W}_{1}}}$, then $\mathbf{y} \not\in \mathcal{P\text{\itshape r}K}(v^{\mu})$ and a fortiori $\mathbf{y} \not\in \mathcal{P\text{\itshape r}N}(v^{\mu})$.
\end{theorem}
\begin{proof}
  From the payoff equivalence classes $[\vec{\gamma}]$ and $[\vec{\gamma}_{1}]$ we derive the corresponding matrices $\boldsymbol{\EuScript{W}} := \boldsymbol{\mathcal{V}}^{\top}\, \boldsymbol{\EuScript{U}}$ and $\boldsymbol{\EuScript{W}}_{1} := \boldsymbol{\mathcal{V}}_{1}^{\top}\, \boldsymbol{\EuScript{U}}$, respectively. By assumption, it is $\Delta \in \mathcal{N}_{\boldsymbol{\EuScript{W}}} \backslash \mathcal{N}_{\boldsymbol{\EuScript{W}_{1}}}$ satisfied. From this argument, we can express the vector of unbalanced excesses $\vec{\xi}^{\,v^{\mu}}$ at $\mathbf{y}$ by 
\begin{equation*}
  \vec{\xi}^{\,v^{\mu}}=\boldsymbol{\mathcal{V}}_{1}^{\top}\,\overline{e}^{\mu}= \boldsymbol{\mathcal{V}}_{1}^{\top}\,(v^{\mu}-\overline{\mathbf{y}})=\boldsymbol{\mathcal{V}}_{1}^{\top}\,( v -\overline{\mathbf{y}} + \mu \cdot v^{\Delta}) = \vec{\xi}^{\,v} + \mu \cdot \boldsymbol{\mathcal{V}}_{1}^{\top}\,v^{\Delta} = \mu \cdot \boldsymbol{\mathcal{V}}_{1}^{\top}\,v^{\Delta}\not= \mathbf{0}.
\end{equation*}
Observe that $\vec{\xi}^{\,v} = \boldsymbol{\mathcal{V}}_{1}^{\top}\,(v-\overline{\mathbf{y}}) = \mathbf{0}$, since vector $\mathbf{y} \in [\vec{\gamma}_{1}]$ is a pre-kernel element of game $v$. However, due to $\Delta \in \mathcal{N}_{\boldsymbol{\EuScript{W}}} \backslash \mathcal{N}_{\boldsymbol{\EuScript{W}_{1}}}$, we obtain $\boldsymbol{\mathcal{V}}_{1}^{\top}\,v^{\Delta} \not= \mathbf{0}$, it follows that $\mathbf{y} \not\in \mathcal{P\text{\itshape r}K}(v^{\mu})$. The conclusion follows that $\mathbf{y} \not\in \mathcal{P\text{\itshape r}N}(v^{\mu})$ must hold.
\end{proof}

\begin{theorem}
  \label{thm:nprnmu3}
 Let $\langle\, N, v\, \rangle$ be a TU game that has a non unique pre-kernel such that $\mathbf{x} \in \mathcal{P\text{\itshape r}K}(v)\backslash \mathcal{P\text{\itshape r}N}(v)$ and $\mathbf{x} \in [\vec{\gamma}]$. If $\langle\, N, v^{\mu}\, \rangle$ is a related game of $v$ with $\mu \not=0$ derived from $\mathbf{x}$ such that $\mathbf{x} \in \mathcal{P\text{\itshape r}K}(v^{\mu}) \cap [\vec{\gamma}]$ holds, then $\mathbf{x} \not\in \mathcal{P\text{\itshape r}N}(v^{\mu})$.
\end{theorem}

\begin{proof}
   According to our assumption $\mathbf{x}$ is not the pre-nucleolus of game $v$, this implies that there exists some $\psi \in \mathbb{R}$ such that $\mathcal{D}^{v}(\psi,\mathbf{x}) \not= \emptyset$ is not balanced. Recall that the set of lexicographically smallest most effective coalitions $\mathcal{S}^{v}(\mathbf{x})$ has not changed under $v^{\mu}$, since $\mathbf{x}$ is a pre-kernel element of game $v^{\mu}$ which still belongs to the payoff equivalence class $[\vec{\gamma}]$. Then exists a bound $\epsilon > 0$ within the maximum surpluses can be varied without effecting the pre-kernel property of $\mathbf{x}$. Thus, we get $\mathcal{D}^{v}(\psi,\mathbf{x}) = \mathcal{D}^{v}(\psi - 2\,\epsilon,\mathbf{x}) \not= \emptyset$ is satisfied. Then $e^{v}(S,\mathbf{x}) - \epsilon \le e^{v}(S,\mathbf{x}) + \mu\cdot v^{\Delta}(S) \le e^{v}(S,\mathbf{x}) + \epsilon$ for all $S \subseteq N$, therefore, this implies $\mathcal{D}^{v^{\mu}}(\psi - \epsilon,\mathbf{x}) = \mathcal{D}^{v}(\psi,\mathbf{x})$. The set $\mathcal{D}^{v^{\mu}}(\psi - \epsilon,\mathbf{x})$ is not balanced, we conclude that $\mathbf{x} \not\in \mathcal{P\text{\itshape r}N}(v^{\mu})$. 
\end{proof}
   
\begin{theorem}
  \label{thm:siva3}
  Assume that the payoff equivalence class $[\vec{\gamma}]$ induced from TU game $\langle\, N, v\, \rangle$ has non-empty interior. In addition, assume that the pre-kernel of game $\langle\, N, v\, \rangle$ constitutes a line segment such that $\mathbf{x} \in \mathcal{P\text{\itshape r}N}(v) \cap \partial\overline{[\vec{\gamma}]}$, $\mathcal{P\text{\itshape r}K}(v) \cap \overline{[\vec{\gamma}_{1}]}$, and $\mathbf{x} \in \mathcal{P\text{\itshape r}K}(v^{\mu}) \cap [\vec{\gamma}]$ is satisfied, then the pre-kernel $\mathcal{P\text{\itshape r}K}(v^{\mu})$ of a related TU game $\langle\, N, v^{\mu}\, \rangle$ with $\mu \not=0$ derived from $\mathbf{x}$ is at least disconnected, otherwise unique. 
\end{theorem}
\begin{proof}
  In the fist step, we have simply to establish that for game $v^{\mu}$ the pre-imputations lying on the part of the line segment included in payoff equivalence class $[\vec{\gamma}_{1}]$ under game $v$ will loose their pre-kernel properties due to the change in the game parameter. In the second step, we have to show that the pre-nucleolus $\mathbf{x}$ under game $v$ is also the pre-nucleolus of the related game $v^{\mu}$.   
   \begin{enumerate}
  \item First notice that the payoff equivalence class $[\vec{\gamma}]$ has full dimension in accordance with its non-empty interior condition. This implies that the vector $\mathbf{x}$ must be the sole pre-kernel element in $\overline{[\vec{\gamma}]}$ (cf. with the proof of Theorem 7.8.1 in~\citet{mei:13}). By our hypothesis, it is even a boundary point of the payoff equivalence class under game $v$. Moreover, it must hold $[\vec{\gamma}] \nsim [\vec{\gamma}_{1}]$, since the rank of the induced matrix $\mathbf{E}^{\top}$ is $n$, and that of $\mathbf{E}^{\top}_{1}$ is $n-1$, therefore, we have $E^{\top}_{1} \not= E^{\top}\,X$ for all $X \in \text{GL}^{+}(n)$. 

In the next step, we select an arbitrary pre-kernel element from $\mathcal{P\text{\itshape r}K}(v) \cap \overline{[\vec{\gamma}_{1}]}$, say $\mathbf{y}$. By hypothesis, there exists a related game $v^{\mu}$ of $v$ such that $\mathbf{x} \in \mathcal{P\text{\itshape r}K}(v^{\mu}) \cap [\vec{\gamma}]$ holds, that is, there is no change in matrix $\mathbf{E}$ and vector $\vec{\alpha}$ implying $h^{v^{\mu}}(\mathbf{x})=h_{\gamma}^{v^{\mu}}(\mathbf{x})=0$. This implies that for game $v^{\mu}$ the payoff equivalence class $[\vec{\gamma}]$ has been enlarged in such a way that we can inscribe an ellipsoid with maximum volume $\varepsilon := \{\mathbf{y}^{\prime}\, \arrowvert h^{v^{\mu}}_{\gamma}(\mathbf{y}^{\prime}) \le \bar{c} \}$, whereas $h_{\gamma}^{v^{\mu}}$ is of type~\eqref{eq:objf2} and $\bar{c} > 0$ (cf. Lemma~\ref{lem:repl_prk}.). It should be obvious that element $\mathbf{x}$ is an interior point of $\varepsilon$, since $\mathbf{x} = M(h_{\gamma}^{v^{\mu}}) \subset \varepsilon \subset [\vec{\gamma}]$. We single out a boundary point $\mathbf{x}^{\prime}$ in $\partial\overline{[\vec{\gamma}]}$ under game $v^{\mu}$ which was a pre-kernel element under game $v$, and satisfying after the parameter change the following properties: $\mathbf{x}^{\prime} \in \partial\overline{[\vec{\gamma}]} \cap \overline{[\vec{\gamma}_{1}]}$ with $\mathbf{x}^{\prime}=\mathbf{x} + \mathbf{z}$, and $\mathbf{z} \not= \mathbf{0}$. This is possible due to the fact that the equivalence class $[\vec{\gamma}]$ has been enlarged at the expense of equivalence class $[\vec{\gamma}_{1}]$, which has shrunk or shifted by the change in the game parameter. Observe now that two cases may happen, that is, either $\mathbf{x}^{\prime} \in \varepsilon$ or $\mathbf{x}^{\prime} \notin \varepsilon$. In the former case, we have $h_{\gamma}^{v^{\mu}}(\mathbf{x}^{\prime}) = h^{v^{\mu}}(\mathbf{x}^{\prime}) =  h_{\gamma_{1}}^{v^{\mu}}(\mathbf{x}^{\prime}) = \bar{c} > 0$, and in the latter case, we have $h_{\gamma}^{v^{\mu}}(\mathbf{x}^{\prime}) = h^{v^{\mu}}(\mathbf{x}^{\prime}) =  h_{\gamma_{1}}^{v^{\mu}}(\mathbf{x}^{\prime}) > \bar{c} > 0 = h^{v}(\mathbf{x}^{\prime}) = h_{\gamma_{1}}^{v}(\mathbf{x}^{\prime})$.

From $h_{\gamma_{1}}^{v^{\mu}}(\mathbf{x}^{\prime}) > 0$, and notice that the vector of unbalanced excesses at $\mathbf{x}^{\prime}$ is denoted as $\vec{\xi}^{\,v^{\mu}}$, we derive the following relationship 
\begin{equation*}
  h_{\gamma_{1}}^{v^{\mu}}(\mathbf{x}^{\prime}) = \Arrowvert\,\vec{\xi}^{\,v^{\mu}} \,\Arrowvert^{2} = \Arrowvert\, \vec{\xi}^{\,v} + \mu \cdot \boldsymbol{\mathcal{V}}_{1}^{\top}\,v^{\Delta}\,\Arrowvert^{2} = \Arrowvert\,\mu \cdot \boldsymbol{\mathcal{V}}_{1}^{\top}\,v^{\Delta} \,\Arrowvert^{2} = \mu^{2} \cdot \Arrowvert\, \boldsymbol{\mathcal{V}}_{1}^{\top}\,v^{\Delta} \,\Arrowvert^{2} > 0, 
\end{equation*}
with $\mu \not=0$. Thus, we have $\boldsymbol{\mathcal{V}}_{1}^{\top}\,v^{\Delta} \not= \mathbf{0}$, and therefore $\Delta \in \mathcal{N}_{\boldsymbol{\EuScript{W}}} \backslash \mathcal{N}_{\boldsymbol{\EuScript{W}_{1}}}$. Observe that $\vec{\xi}^{\,v} = \boldsymbol{\mathcal{V}}_{1}^{\top}\,(v-\overline{\mathbf{x}^{\prime}}) = \mathbf{0}$, since vector $\mathbf{x}^{\prime} \in \overline{[\vec{\gamma}_{1}]}$ is a pre-kernel element of game $v$. Take the vector $\mathbf{y} \in [\vec{\gamma}_{1}]$ from above that was on the line segment as vector $\mathbf{x}^{\prime}$ under game $v$ which constituted a part of the pre-kernel of game $v$, we conclude that $\mathbf{y} \not\in \mathcal{P\text{\itshape r}K}(v^{\mu})$ in accordance with $\boldsymbol{\mathcal{V}}_{1}^{\top}\,v^{\Delta} \not= \mathbf{0}$. 

 \item By our hypothesis, $\mathbf{x}$ is the pre-nucleolus of game $v$, and an interior point of equivalence class $[\vec{\gamma}]$ of the related game $v^{\mu}$. Using a similar argument as under (1) we can inscribe an ellipsoid with maximum volume $\varepsilon$, whereas $h_{\gamma}^{v^{\mu}}$ is of type~\eqref{eq:objf2} and $\bar{c} > 0$. In view of the assumption that $\mathbf{x}$ is also pre-kernel element of game $v^{\mu}$, we can draw the conclusion that the set of lexicographically smallest most effective coalitions $\mathcal{S}(\mathbf{x})$ has not changed under $v^{\mu}$. But then, we have $\mu\cdot v^{\Delta} \in [-\mathsf{C},\mathsf{C}]^{p^{\prime}}$. In addition, there exists a $\bar{\psi} \ge  \psi^{*}$ s.t. $\mathcal{S}(\mathbf{x}) \subseteq \mathcal{D}^{v}(\bar{\psi},\mathbf{x})$, that is, it satisfies Property I of~\citet{kohlb:71}. Moreover, matrix $\mathbf{E}^{\top}$ induced from $\mathcal{S}(\mathbf{x})$ has full rank, therefore, the column vectors of matrix $\mathbf{E}^{\top}$ are a spanning system of $\mathbb{R}^{n}$. Hence, we get $span\,\{\mathbf{1}_{S}\,\arrowvert\, S \in \mathcal{S}(\mathbf{x}) \} = \mathbb{R}^{n}$ as well, which implies that matrix $[\mathbf{1}_{S}]_{S \in \mathcal{S}(\mathbf{x})}$ has rank $n$, the collection $\mathcal{S}(\mathbf{x})$ must be balanced. In accordance with vector $\mathbf{x}$ as the pre-nucleolus of game $v$, we can choose the largest $\psi \in \mathbb{R}$ s.t. $\emptyset \not= \mathcal{D}^{v}(\psi,\mathbf{x}) \subseteq \mathcal{S}(\mathbf{x})$ is valid, which is a balanced set. Since $\mathsf{C} > 0$,  the set $\mathcal{D}^{v}(\psi - 2\,\mathsf{C},\mathbf{x}) \not=\emptyset$ is balanced as well. Now observe that $e^{v}(S,\mathbf{x}) -\mathsf{C} \le e^{v}(S,\mathbf{x}) + \mu\cdot v^{\Delta}(S) \le e^{v}(S,\mathbf{x}) + \mathsf{C}$ for all $S \subseteq N$. This implies $\mathcal{D}^{v}(\psi,\mathbf{x}) \subseteq \mathcal{S}(\mathbf{x}) \subseteq \mathcal{D}^{v^{\mu}}(\psi - \mathsf{C},\mathbf{x}) \subseteq \mathcal{D}^{v}(\psi - 2\,\mathsf{C},\mathbf{x})$, hence, $ \mathcal{D}^{v^{\mu}}(\psi - \mathsf{C},\mathbf{x})$ is balanced. To conclude, let $c \in [-\mathsf{C},\mathsf{C}]$, and from the observation $\lim_{c \uparrow 0}\, \mathcal{D}^{v^{\mu}}(\psi + c,\mathbf{x}) = \mathcal{D}^{v^{\mu}}(\psi,\mathbf{x}) \supseteq \mathcal{D}^{v}(\psi,\mathbf{x}) $, we draw the implication $\mathbf{x} = \nu(N,v^{\,\mu})$. 
\end{enumerate}
Finally, recall that the vector $\mathbf{x}$ is also the unique minimizer of function $h_{\gamma}^{v^{\mu}}$, which is an interior point of payoff equivalence class $[\vec{\gamma}]$, therefore the pre-kernel of the related game $v^{\mu}$ can not be connected. Otherwise the pre-kernel of the game consists of a single point.  
\end{proof}

\begin{corollary}
  \label{cor:prnmu3}
 Let $\langle\, N, v\, \rangle$ be a TU game that has a non single-valued pre-kernel such that $\mathbf{x} \in \mathcal{P\text{\itshape r}N}(v) \cap \partial\overline{[\vec{\gamma}]}$ and let $\langle\, N, v^{\mu}\, \rangle$ be a related game of $v$ derived from $\mathbf{x}$, whereas $\mathbf{x} \in int\,[\vec{\gamma}]_{v^{\mu}}$, then $\mathbf{x} = \nu(N,v^{\,\mu})$.
\end{corollary}

\section{Concluding Remarks}
\label{sec:rem}
In this paper we have established that the set of related games derived from a default game with an unique pre-kernel must also possess this pre-kernel element as its single pre-kernel  point. Moreover, we have shown that the pre-kernel correspondence in the game space restricted to the convex hull comprising the default and related games is single-valued and constant, and therefore continuous. Although, we could provide some sufficient conditions under which the pre-nucleolus of a default game -- whereas the pre-kernel constitutes a line segment -- induces at least a disconnected pre-kernel for the set of related games, it is, however, still an open question if it is possible to obtain from a game with a non-unique pre-kernel some related games that have an unique pre-kernel. In this respect, the knowledge of more general conditions that preserve the pre-nucleolus property is of particular interest.

Even though, we have not provided a new set of game classes with a sole pre-kernel element, we nevertheless think that the presented approach is also very useful to bring forward our knowledge about the classes of transferable utility games where the pre-kernel coalesces with the pre-nucleolus. To answer this question, one need just to select boundary points of the convex cone of the class of convex games to enlarge the convex cone within the game space to identify game classes that allow for a singleton pre-kernel.         

\pagestyle{scrheadings} \chead{\empty}  
\footnotesize
\bibliography{siva_prk}

\begin{thebibliography}{12}
\providecommand{\natexlab}[1]{#1}
\providecommand{\url}[1]{\texttt{#1}}
\expandafter\ifx\csname urlstyle\endcsname\relax
  \providecommand{\doi}[1]{doi: #1}\else
  \providecommand{\doi}{doi: \begingroup \urlstyle{rm}\Url}\fi

\bibitem[Get{\'a}n et~al.(2012)Get{\'a}n, Izquierdo, Montes, and
  Rafels]{getraf:12}
J.~Get{\'a}n, J.~Izquierdo, J.~Montes, and C.~Rafels.
\newblock {The Bargaining Set and the Kernel of Almost-Convex Games}.
\newblock Technical report, University of Barcelona, Spain, 2012.

\bibitem[Kohlberg(1971)]{kohlb:71}
E.~Kohlberg.
\newblock {On the Nucleolus of a Characteristic Function Game}.
\newblock \emph{SIAM Journal of Applied Mathematics}, 20:\penalty0 62--66,
  1971.

\bibitem[Martinez-Legaz(1996)]{mart:96}
J-E. Martinez-Legaz.
\newblock {Dual Representation of Cooperative Games based on Fenchel-Moreau
  Conjugation}.
\newblock \emph{{Optimization}}, 36:\penalty0 291--319, 1996.

\bibitem[Maschler et~al.(1972)Maschler, Peleg, and Shapley]{MPSh:72}
M.~Maschler, B.~Peleg, and L.S. Shapley.
\newblock {The Kernel and Bargaining Set for Convex Games}.
\newblock \emph{{International Journal of Game Theory}}, 1:\penalty0 73--93,
  1972.

\bibitem[Maschler et~al.(1979)Maschler, Peleg, and Shapley]{MPSh:79}
M.~Maschler, B.~Peleg, and L.~S. Shapley.
\newblock {Geometric Properties of the Kernel, Nucleolus, and Related Solution
  Concepts}.
\newblock \emph{{Mathematics of Operations Research}}, 4:\penalty0 303--338,
  1979.

\bibitem[Meinhardt(2013)]{mei:13}
H.~I. Meinhardt.
\newblock \emph{{The Pre-Kernel as a Tractable Solution for Cooperative Games:
  An Exercise in Algorithmic Game Theory}}, volume~45 of \emph{{Theory and
  Decision Library: Series C}}.
\newblock Springer Publisher, Heidelberg/Berlin, 2013.

\bibitem[Meinhardt(2015{\natexlab{a}})]{mei:10a}
H.~I. Meinhardt.
\newblock {TuGames: A Mathematica Package for Cooperative Game Theory}.
\newblock Version 2.4, Karlsruhe Institute of Technology (KIT), Karlsruhe,
  Germany, 2015{\natexlab{a}}.
\newblock URL \url{http://library.wolfram.com/infocenter/MathSource/5709}.

\bibitem[Meinhardt(2015{\natexlab{b}})]{mei:11}
H.~I. Meinhardt.
\newblock {MatTuGames: A Matlab Toolbox for Cooperative Game Theory}.
\newblock Version 0.7, Karlsruhe Institute of Technology (KIT), Karlsruhe,
  Germany, 2015{\natexlab{b}}.
\newblock URL
  \url{http://www.mathworks.com/matlabcentral/fileexchange/35933-mattugames}.


\bibitem[Meseguer-Artola(1997)]{mes:97}
A.~Meseguer-Artola.
\newblock {Using the Indirect Function to characterize the Kernel of a
  TU-Game}.
\newblock Technical report, Departament d'Economia i d'Hist{\`o}ria
  Econ{\`o}mica, Universitat Aut{\`o}noma de Barcelona, Nov. 1997.
\newblock mimeo.

\bibitem[Peleg and Sudh{\"o}lter(2007)]{pel_sud:07}
B.~Peleg and P.~Sudh{\"o}lter.
\newblock \emph{{Introduction to the Theory of Cooperative Games}}, volume~34
  of \emph{{Theory and Decision Library: Series C}}.
\newblock {Springer-Verlag}, 2 edition, 2007.

\bibitem[Rockafellar(1970)]{Rocka:70}
R.~Rockafellar.
\newblock \emph{{Convex Analysis}}.
\newblock {Princeton University Press}, {Princeton}, 1970.

\bibitem[Shapley(1971)]{Shapley:71}
L.~S. Shapley.
\newblock {Cores of Convex Games,}.
\newblock \emph{{International Journal of Game Theory}}, 1:\penalty0 11--26,
  1971.

\end{thebibliography}

\end{document}